\documentclass[11pt]{amsart}

\usepackage{tikz}  % For drawings

\usepackage{amssymb}
\usepackage{epsfig}
\usepackage{comment}
\usepackage{amsmath}
\usepackage{caption}
\usepackage{subcaption}
\usepackage[section]{placeins}
\usepackage{graphicx}
\usepackage{units}
\usepackage{color}

\definecolor{mygray}{gray}{.6}

\newcommand\duality[2]{\langle #1, #2\rangle_{H^{-1/2}(\partial\Omega),H^{1/2}(\partial\Omega)}}

\newcommand\dualitysh[2]{\langle #1, #2\rangle}
\newtheorem{thm}{Theorem}[section]

\newtheorem{lem}[thm]{Lemma}
\newtheorem{lemma}[thm]{Lemma}
\newtheorem{theorem}[thm]{Theorem}
\newtheorem{cor}[thm]{Corollary}

\theoremstyle{definition}
\newtheorem{definition}[thm]{Definition}
\newtheorem{example}[thm]{Example}

\newcommand\setD{\mathcal{D}}
\newcommand\setE{\mathcal{E}}
\newcommand\topoD{\Delta}
\numberwithin{equation}{section}
\theoremstyle{remark}
\newtheorem{remark}[thm]{Remark}
\newcommand\eps{\epsilon}

\newcommand\R{\mathbb{R}}
\newcommand\N{\mathbb{N}}
\newcommand\Z{\mathbb{Z}}
\newcommand\C{\mathbb{C}}

\newcommand\curl{\mathrm{curl\,}}
\newcommand\Div{\mathrm{div\,}}
\newcommand\sym{\mathrm{sym}}
\newcommand\loc{\mathrm{loc}}
\newcommand\calL{\mathcal L}
\newcommand\calH{\mathcal H}

\newcommand\calA{\mathcal A}
\newcommand\weakto{\rightharpoonup}
\newcommand{\weakstarto}{\ensuremath{\overset{\ast}{\rightharpoonup}}}

\newcommand\Id{\mathrm{Id}}
\newcommand\tr{\mathop{\mathrm{tr}}}
\newcommand\Span{\mathop{\mathrm{span}}}

\begin{document}
%{\bf Notation: $\Delta-\lim$, $(D,E)$ or $(\calD, \calE)$.}
\title{Data-Driven Problems in Elasticity}

\author[S.~Conti, S.~M\"uller and M.~Ortiz]
{S.~Conti$^1$, S.~M\"uller$^{1,2}$ and M.~Ortiz$^{1,2,3}$}

\address
{
    $^1$ Institut f\"ur Angewandte Mathematik, Universit\"at Bonn,
    Endenicher Allee 60,
    53115 Bonn, Germany.
}
\address
{
    $^2$ Hausdorff Center for Mathematics,
    Endenicher Allee 60,
    53115 Bonn, Germany.
}

\address
{
    $^3$ Division of Engineering and Applied Science,
    California Institute of Technology,
    1200 E.~California Blvd., Pasadena, CA 91125, USA.
}

%\email{ortiz@caltech.edu}
%%\urladdr{http://www.ortiz.caltech.edu/$\sim$ortiz/home.shtml}

\begin{abstract}
We consider a new class of problems in elasticity, referred to as Data-Driven problems, defined on the space of strain-stress field pairs, or phase space. The problem consists of minimizing the distance between a given material data set and the subspace of compatible strain fields and stress fields in equilibrium. We find that the classical solutions are recovered in the case of linear elasticity. We identify conditions for convergence of Data-Driven solutions corresponding to sequences of approximating material data sets. Specialization to constant material data set sequences in turn establishes an appropriate notion of relaxation. We find that relaxation within this Data-Driven framework is fundamentally different from the classical relaxation of energy functions. For instance, we show that in the Data-Driven framework the relaxation of a bistable material leads to material data sets that are not graphs.
\end{abstract}

\maketitle

%\tableofcontents

\def\ie{i.~e.}
\def\ae{a.~e.}
\def\eg{e.~g.}

\section{Introduction}

To date, the prevailing and classical scientific paradigm in materials science has been to calibrate empirical material models using observational data and then use the calibrated material models to define initial-boundary value problems. This process of modelling inevitably adds error and uncertainty to the solutions, especially in systems with high-dimensional phase spaces and complex material behavior. The modelling error and uncertainty arises mainly from imperfect knowledge of the functional form of the material laws, the phase space in which they are defined, and from scatter and noise in the experimental data.

Against this classical backdrop, remarkable advances in experimental science over the past few decades, including atomic probing, digital imaging, microscopy,  and diffraction methods, have radically changed the nature of materials science and engineering from {\sl data-starved fields} to, increasingly, {\sl data-rich fields}. In addition, multiscale analysis presently allows for accurate and reliable {\sl Data Mining} of material data from lower scales. These advances open the way for the application to materials science of concepts from the emerging field of {\sl Data Science} (cf., e.~g., \cite{WD002, WD023}). Specifically, the abundance of data suggests the possibility of a new scientific paradigm, to be referred to as the {\sl Data-Driven paradigm}, consisting of reformulating the classical initial-boundary-value problems directly from material data, thus bypassing the empirical material modelling step of traditional materials science and engineering altogether. In this manner, material modelling empiricism, error and uncertainty are eliminated entirely and no loss of experimental information is incurred. Data Science currently influences primarily non-STEM fields such as marketing, advertising, finance, social sciences, security, policy, and medical informatics, among others. By contrast, the full potential of Data Science as it relates to problems in materials science and engineering has yet to be explored and realized.

A mathematical connection between Data-Science and materials science can be forged as follows. We begin by noting that the field theories of science have a common general structure. Perhaps the simplest field theory is potential theory, which arises in the context of Newtonian mechanics, hydrodynamics, electrostatics, diffusion, and other fields of application. In this case, the field $u : \Omega \to \mathbb{R}$ that describes the state of the system is scalar. The localization law that extracts from $u$ the local state at a given material point $x \in \Omega$ is
\begin{equation}\label{thla1I}
    \epsilon(x) = \nabla u(x),
\end{equation}
\ie, the localization operator is simply the gradient of the field, together with essential boundary conditions of the Dirichlet type. Evidently, (\ref{thla1I}) constrains the field $\epsilon : \Omega \to \mathbb{R}^n$ to be a gradient, \ie, to be {\sl compatible}. The corresponding conjugate variable is the flux $\sigma : \Omega \to \mathbb{R}^n$. The flux satisfies the conservation equation
\begin{equation}\label{2RLefr}
    \nabla\cdot\sigma(x) + \rho(x) = 0,
\end{equation}
where $\nabla\cdot$ is the divergence operator and $\rho$ is a source density, together with natural boundary conditions of the Neumann type. The pair $z(x) = (\epsilon(x),\sigma(x))$ describes the local state of the system at the material point $x$ and the function $z$ maps $\Omega$ to the {\sl local phase space} ${Z}_{\rm loc} =\mathbb{R}^n\times\mathbb{R}^n$. The collection of state functions $z : \Omega \to {Z}_{\rm loc}$ defines the {\sl global state space} ${Z}$. We note that the phase space, compatibility and conservation laws are universal, \ie, material independent. We may thus define a material-independent {\sl constraint set} ${\setE} \subset {Z}$ to be the set of states $z = (\epsilon,\sigma)$ consistent with the compatibility and conservation laws (\ref{thla1I}) and (\ref{2RLefr}), respectively, as well as  corresponding essential and natural boundary conditions thereof.

By way of contrast, the material law
\begin{equation}\label{j8EmOa}
    \sigma(x) = \sigma(\epsilon(x) )
\end{equation}
that closes the equations is often only known imperfectly through a {\sl local material data set} ${\setD}_{\rm loc}$ in local phase space ${Z}_{\rm loc}$ that collects the totality of our empirical knowledge of the material. A typical local material data set consists of a finite number of local states, ${\setD}_{\rm loc} = \{(\epsilon_i,\sigma_i),\ i=1,\dots,N\}$. A corresponding {\sl global material data set} ${\setD} \subset {Z}$ can then be identified with
\begin{equation}
    {\setD} = \{z \in {Z} : \ z(x) \in {\setD}_{\rm loc} \}
\end{equation}
Evidently, for a material data set of this type, the intersection ${\setD} \cap {\setE}$ is likely to be empty, \ie, there may be no points in the material data set that are compatible and satisfy the conservation laws, even in cases when solutions could reasonably be expected to exist. It is, therefore, necessary to replace the overly-rigid characterization of the solution set ${S} = {\setD} \cap {\setE}$ by a suitable relaxation thereof. One such relaxed formulation \cite{WD005} consists of accepting as the best possible solution the state $z$ in the material data set ${\setD}$ that is closest to the constraint set ${\setE}$, \ie, which is closest to satisfying compatibility and the conservation law. Closeness is understood in terms of some appropriate distance $d$ defined on the state space ${Z}$. The Data-Driven solution set is, then,
\begin{equation}\label{trLu7O}
    {S} = {\rm argmin} \{ d(z,{\setD}),\ z \in {\setE} \} .
\end{equation}
We emphasize that Data-Driven solutions are determined directly from the material data set ${\setD}$ and that no attempt is made at modeling, \ie, at approximating the local material data sets ${\setD}_{\rm loc}$ by means of a graph of the form (\ref{j8EmOa}).

In general, we consider systems whose state is characterized by points $z$ in a metric space $({Z},d)$. The compatibility, conservation and boundary constraints acting on the system have the effect of restricting its state to a subset ${\setE} \subset {Z}$. In addition, the behavior of the material is described by a material data set ${\setD} \subset {Z}$. The corresponding Data-Driven problem is then (\ref{trLu7O}). A number of fundamental questions arise in connection with this new class of problems. It is clear that the range of Data-Driven problems is larger than that of classical problems since the local material data sets, even if they define a curve in phase space, need not be a graph. It is therefore of interest to know if the classical solutions are recovered when the local material data sets ${\setD}_{\rm loc}$ are a graph $(\epsilon, \sigma(\epsilon))$. Secondly, it is of interest to elucidate the dependence of the Data-Driven solutions on the material data sets. In particular, we wish to ascertain the convergence properties of sequences $(z_h)$ of Data-Driven solutions generated by sequences $({\setD}_h)$ of material data sets. Finally, a central question of analysis concerns the existence of Data-Driven solutions and, in cases of non-attainment, the relaxed form of the Data-Driven problem.

\begin{figure}[ht]
    \includegraphics[width=0.66\linewidth]{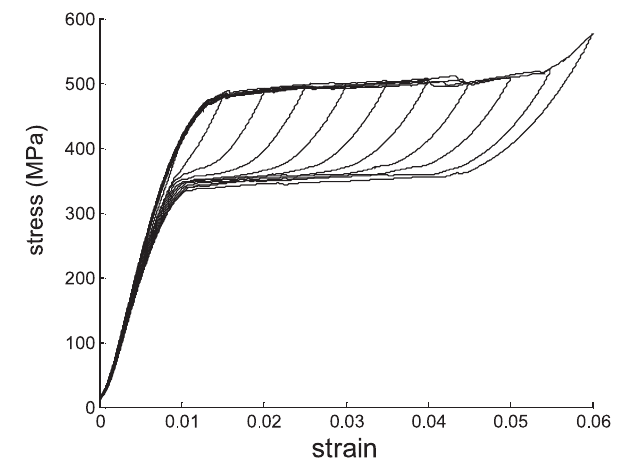}
    \caption{Experimental hysteresis loops of Ni–Ti shape memory alloy (SMA) wires at various strain amplitudes. Figure
    reproduced with permission from \cite[Fig. 3]{Ren2007},
    \copyright\ IOP Publishing.  All rights reserved.} \label{KlUql9}
\end{figure}

We investigate these questions in the special case of elasticity under the linearized kinematics approximation. We start from a situation in which the data set is weakly closed and therefore amenable to more straightforward analysis. We specifically consider the simple case of linear elasticity and seek solutions in the phase space of $L^2$ strain-stress field pairs. The Data-Driven problem then consists of minimizing the distance between a given material data set and the subspace of compatible strain fields and stress fields in equilibrium. In the case of linear elasticity, we find that the classical solutions are recovered. Similar reasoning can be applied to data sets that take the form of a monotone graph in phase space. We also identify conditions for convergence of Data-Driven solutions corresponding to sequences of approximating material data sets. Specialization to constant material data set sequences in turn establishes an appropriate notion of relaxation. We find that relaxation within this Data-Driven framework is fundamentally different from the classical relaxation of energy functions. For instance, we show that in the Data-Driven framework the relaxation of a bistable material leads to material data sets that are not graphs. The relaxed local material data set is reminiscent of the 'flag' sets that are covered by hysteretic loops in tests of shape memory alloy wires, Fig.~\ref{KlUql9}. This similarity suggests a useful role for Data-Driven analysis in connection with the characterization of such materials. These results also illustrate the fact that Data-Driven problems subsume---and are strictly more general than---classical energy minimization problems.

%Luc Tartar has emphasized
The point of view that the nonlinear partial differential equations (pdes) of continuum mechanics can be written as a set of linear pdes (balance laws) and nonlinear pointwise relations between the quantities in the balance laws (constitutive relations) has been emphasized by Luc Tartar since the 1970s. He has also stressed the importance to understand the relation between the linear pde constraints and the pointwise constraints in the context of \emph{weak convergence.} Weak convergence arises naturally in the context of effective properties of heterogeneous media (homogenization), existence, optimal control and in elucidating which quantities can effectively be measured.

Tartar’s  method of {\sl compensated compactness} (developed in collaboration with Fran\c{c}ois Murat) provides a powerful mathematical framework for analyzing which nonlinear relations are stable under weak convergence in the presence of linear pde constraints. For an early exposition of these ideas, see \cite{Tartar1979}. For further developments, see \cite{Tartar1985,Tartar1990_Hmeas} and the monograph \cite{Tartar2009Buch}. A key result in the theory of compensated compactness is the {\sl div-curl Lemma} \cite{Murat1978, Tartar1978, Tartar1979, Murat1981, Tartar1982}. This lemma is closely related to the weak continuity of determinants, which plays a fundamental role in nonlinear elasticity and quasiconformal geometry \cite{Reshetnyak1967, Reshetnyak1968, Ball1977}.

Many questions arising in the study of effective properties of composites may also be formulated in the present framework. For example, two-material composites in a geometrically linear setting correspond to the case in which the set $\setD_\loc$ is the union of two graphs in the strain-stress phase space. The computation of its relaxation $\overline \setD_\loc$ is then related to the computation of the $G$-closure. We refer to \cite{Willis1981, MuratTartar1985, Tartar1986, FrancfortMurat1986, KohnStrang1986I, KohnStrang1986II, NesiMilton1991, Nesi1995} for early results and to the monographs \cite{Cherkaev2000Buch, Allaire2002Buch, Milton2002, Tartar2009Buch} for modern reviews of the subject from different viewpoints.

In this paper, we introduce a new approach for the variational study of materials that is not based on energy minimization but instead accords the equilibrium equations and material data set a central role. We do not aim for the most general abstract setting. For conceptual clarity we instead focus on two illustrative examples, namely, linear elasticity in Section \ref{seclinear} and the geometrically linear two-well problem in Section \ref{secrelax}. Specifically, we develop a concept of relaxation in the stress-strain phase space that enforces strongly the differential constraints of the stress being divergence-free and the strain being a symmetrized gradient, while the constitutive equation is enforced only in an asymptotic sense. This notion of relaxation can also be couched in terms of $\calA$-quasiconvexity, see \cite[p. 14 and p.100-112]{Dacorogna1982Buch} and \cite{FonsecaMueller1999}, but a detailed elucidation of this connection is beyond the scope of this paper.

It is instructive to briefly compare the present approach to relaxation by energy minimization. By way of example, in Section \ref{sectwowell} we consider the geometrically linear %compatible   %removed SM Aug 16
two-well problem. In the context of energy minimization, the two-well problem entails the study of energy densities that are the minimum of two quadratic forms, \eg, $W(\xi)=\min\{\frac12\C (\xi-a)\cdot(\xi-a), \frac12\C (\xi-b)\cdot(\xi-b)\}$. 
Of central interest is the effective energy that describes the weak lower semicontinuous envelope of the integral functional $\int_\Omega W(Du)dx$. 
There exists a large body of literature on the subject. The geometrically linear case of interest here was addressed early on in \cite{Khachaturyan1967some, KhachaturyanShatalov1969theory, Roitburd69,RN26}. 
The corresponding geometrically nonlinear theory was developed in \cite{BallJames87} in the particular context of martensitic phase transitions. These problems have spawned a large body of research, typically concerned with finding the zero-set of the relaxation. 
The present approach is fundamentally different since, instead of minimizing energy, we allow for all sequences of stresses and strains that obey the compatibility and equilibrium constraints, which results in a larger envelope, see discussion in Section \ref{sectwowell}. 
In a different but related context, the observation that for problems of this type the set of solutions of the equilibrium equations 
is much larger than the set of local minimizers of the energy---and that it indeed contains much wilder solutions---was made 
in particular in \cite{MuellerSverak1998, MuellerSverak2003}.

\section{Linear elasticity}
\label{seclinear}
\subsection{General setting}

\begin{figure}[ht]
	\begin{subfigure}{0.45\textwidth}\caption{} \includegraphics[width=0.99\linewidth]{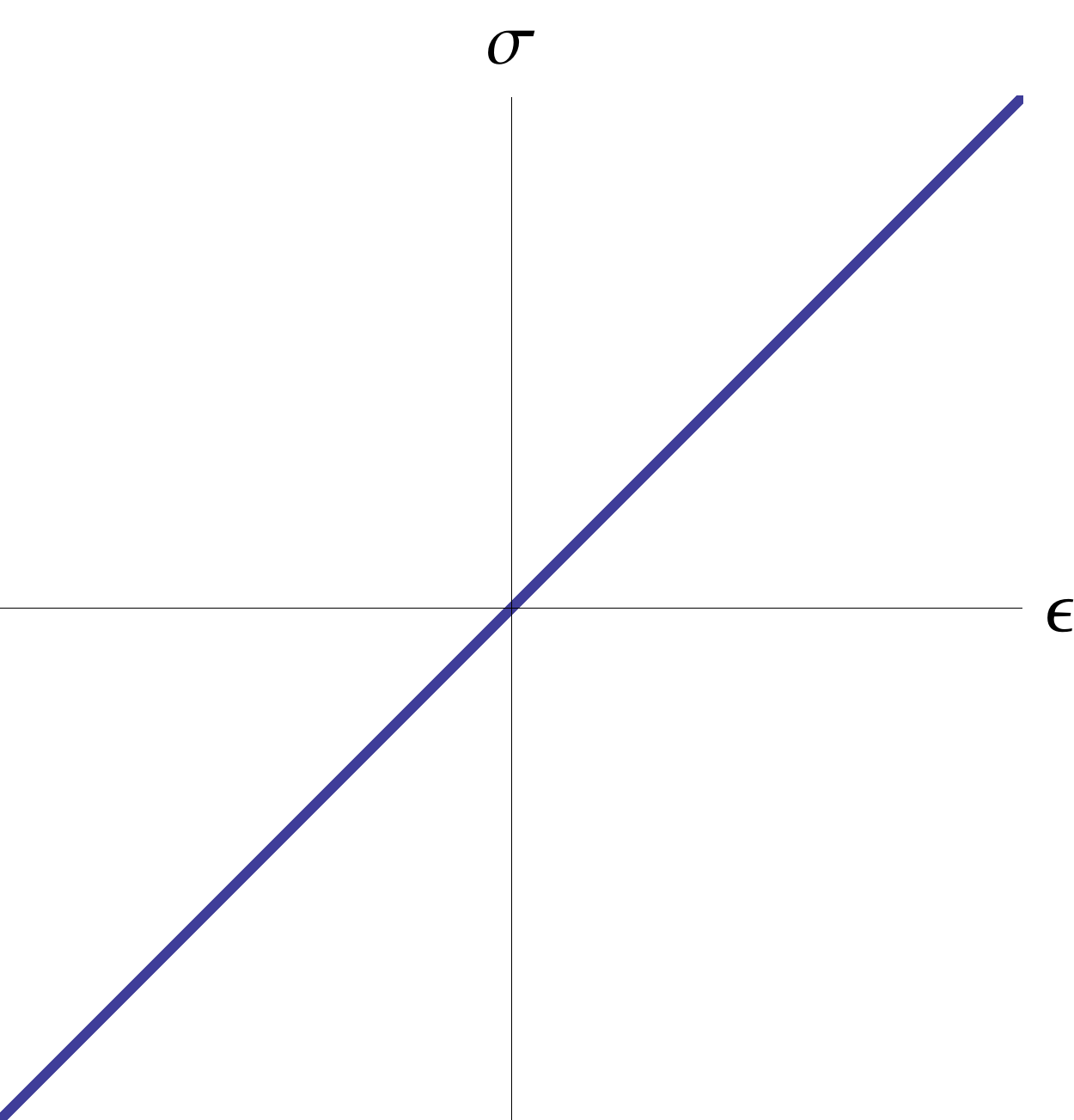}
	\end{subfigure}
    $\quad$
	\begin{subfigure}{0.45\textwidth}\caption{} \includegraphics[width=0.99\linewidth]{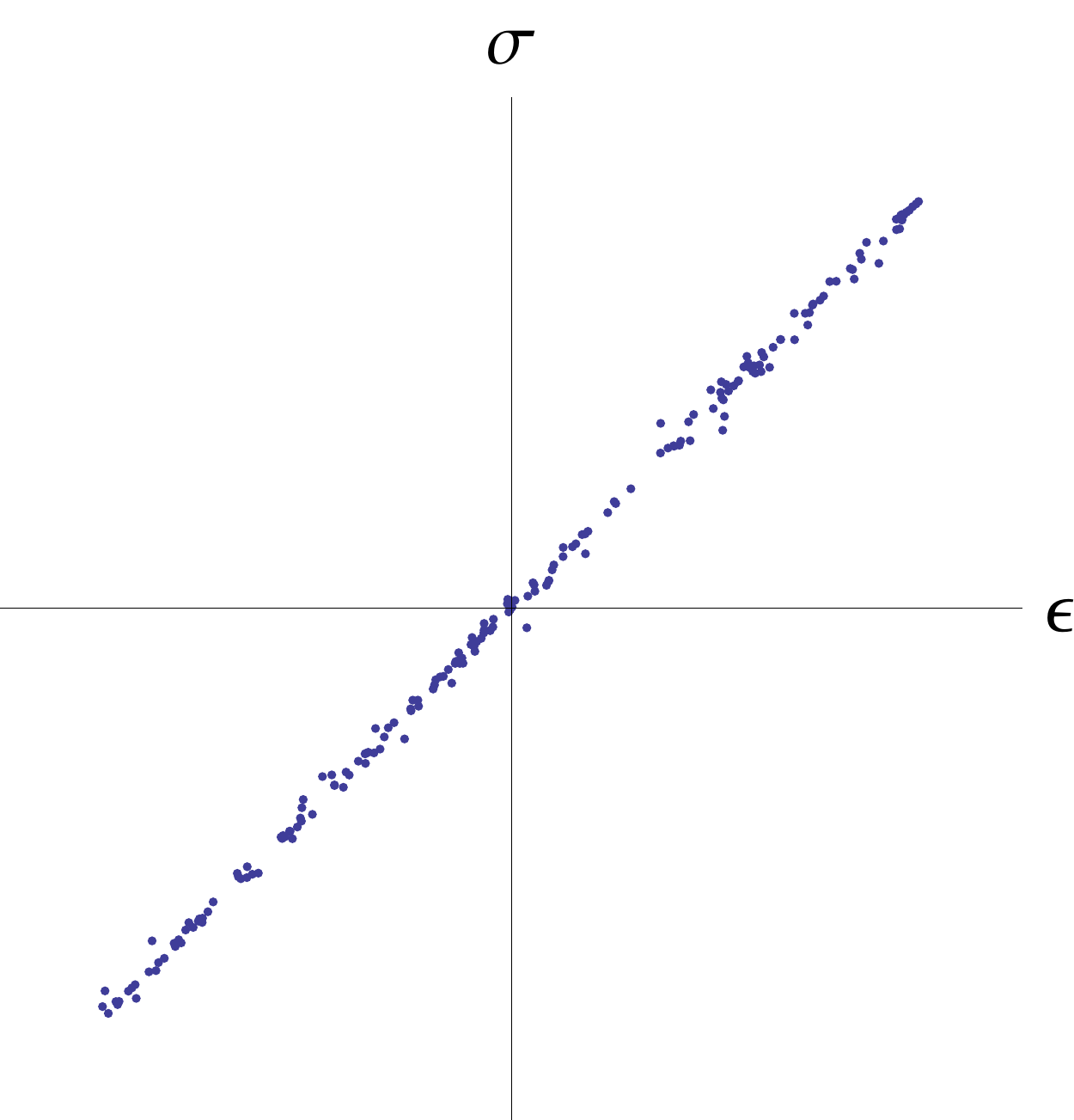}
	\end{subfigure}
	\caption{a) Local material data set for linearized elasticity. b) Sampled local material data set for linearized elasticity.} 	\label{3rlASw}
\end{figure}

We consider an elastic body occupying a bounded, connected Lipschitz set $\Omega \subset \mathbb{R}^{n}$ with state defined by a displacement field $u : \Omega \to \mathbb{R}^{n}$. The corresponding compatibility and equilibrium laws are
\begin{subequations}\label{sIes1A}
\begin{align}
    &
    \epsilon(x) = \frac{1}{2} \big( \nabla u(x) + \nabla u^T(x) \big) ,
    &  \text{in } \Omega ,
    \\ & \label{eqGammaD}
    u(x) = g(x) ,
    &   \text{on } \Gamma_D ,
\end{align}
\end{subequations}
and
\begin{subequations}\label{fRoa1l}
\begin{align}
    &
    {\rm div} \sigma(x) + f(x) = 0 ,
    &
    \text{in } \Omega ,
    \\ &
    \sigma(x) \nu(x) = h(x) ,
    &  \label{eqGammaN}
    \text{on } \Gamma_N ,
\end{align}
\end{subequations}
where $\epsilon \in L^2( \Omega; \mathbb{R}^{{n}\times {n}}_{\rm sym})$ is the strain tensor field and $\sigma\in L^2( \Omega ; \mathbb{R}^{{n}\times {n}}_{\rm sym})$ is the stress tensor field, $f \in L^2(\Omega ; \mathbb{R}^{n})$ are body forces, $g \in H^{1/2}( \partial\Omega ; \mathbb{R}^{n})$ boundary displacements, $h \in H^{-1/2}(\partial\Omega; \mathbb{R}^{n})$ applied tractions and  $\nu:\partial\Omega\to S^{n-1}$ denotes the outer normal.

\def\OmegaRk{(\Omega;\R^{n})}
\def\OmegaRkn{(\Omega;\R^{n\times n})}
\def\randOmegaRk{(\partial\Omega;\R^{n})}
In order to explain the meaning of the boundary conditions (\ref{eqGammaD}) and (\ref{eqGammaN})  we recall some properties of the function spaces involved. It is well known that there exists a unique linear continuous operator $T_0:H^1\OmegaRk\to L^2\randOmegaRk$ such that $T_0 u(x) = u(x)$ for all $u\in C^1(\bar\Omega;\R^{n})$ and $x\in\partial\Omega$. The space $H^{1/2}\randOmegaRk$ is defined as the range of this operator, namely,  $H^{1/2}\randOmegaRk=T_0(H^1\OmegaRk)\subset L^2\randOmegaRk$, with the induced metric, so that $T_0$ is continuous from $H^1\OmegaRk$ to $H^{1/2}\randOmegaRk$. The Dirichlet boundary condition (\ref{eqGammaD}) means that $T_0u(x)=g(x)$ for $\calH^{n-1}$-almost every $x\in\Gamma_D\subset\partial\Omega$.

The space $H^{-1/2}\randOmegaRk$ is defined as the dual of  $H^{1/2}\randOmegaRk$. Following \cite{Temam1979BuchNavierStokes} we define $E\OmegaRkn$ as the set of $w\in L^2\OmegaRkn$ such that $\Div w \in L^2\OmegaRk$. Then, it can be shown that there exists a unique linear continuous operator $T_\nu: E\OmegaRkn\to H^{-1/2}\randOmegaRk$ such that $(T_\nu w)(x) = (w\cdot\nu)(x)$ for all $w\in C^1(\bar\Omega;\R^{n\times n})$ and $\calH^{n-1}$-almost all $x\in\partial\Omega$ (we recall that $\nu$ is the outer normal). Furthermore, it follows that
\begin{equation}\label{eqpartintegrh12hm12}
    \int_\Omega w \cdot Du\, dx + \int_\Omega \Div w \cdot u \,dx = \duality{T_\nu w}{T_0u}.
\end{equation}
for any $w\in E\OmegaRkn$ and $u\in H^1\OmegaRk$. Finally, $T_\nu$ is surjective, in the sense that for any $\theta\in H^{-1/2}\randOmegaRk$ there is $w\in E\OmegaRkn$ such that $T_\nu w=\theta$. Proofs of these results are given, \eg, in \cite[Sect. 1.2 and 1.3]{Temam1979BuchNavierStokes}. The Neumann boundary condition (\ref{eqGammaN}) specifically means that
\begin{equation}
    \duality{T_\nu\sigma}{\psi} = \duality{h}{\psi} ,
\end{equation}
for any $\psi\in H^{1/2}(\partial\Omega;\R^n)$ that obeys $\psi=0$  $\calH^{n-1}$-almost everywhere on $\partial\Omega\setminus\Gamma_N$.

In the remainder of the paper we shall simply write $u$ and $\sigma\nu$ for the traces $T_0u$ and $T_\nu \sigma$. Furthermore, we shall assume throughout that $\Omega \subset \mathbb{R}^{{n}}$ is a connected, open, bounded, nonempty Lipschitz set, and that $\Gamma_D$, $\Gamma_N$ are disjoint open subsets of $\partial\Omega$ with
\begin{equation}\label{eqassGammaDN}
    \overline{\Gamma}_D \cap \overline{\Gamma}_N = \partial\Omega,\hskip3mm
    \calH^{n-1}(\overline\Gamma_N\setminus\Gamma_N)=\calH^{n-1}(\overline\Gamma_D\setminus\Gamma_D)=0,
    \hskip3mm\Gamma_D\ne\emptyset.
 \end{equation}
We remark that, if $\Gamma_D$ and $\Gamma_N$ are Lipschitz subsets of $\partial\Omega$, then extension is possible and it is possible to take  $g \in H^{1/2}( \Gamma_D ; \mathbb{R}^{n})$ and $h \in H^{-1/2}(\Gamma_N; \mathbb{R}^{n})$.

\subsection{Linear elasticity}
\label{jouM0a}

We define the phase space as
\begin{equation}\label{eqdefZlinearelast}
    Z
    =
    L^2(\Omega;\mathbb{R}^{{n} \times {n}}_{\rm sym})
    \times
    L^2(\Omega;\mathbb{R}^{{n}\times {n}}_{\rm sym})
\end{equation}
and metrize it  by the norm
\begin{equation}\label{nIus4o}
    \| z \|
    =
    \left(
        \int_\Omega
            \Big(
                \frac{1}{2} \mathbb{C}\epsilon \cdot \epsilon
                +
                \frac{1}{2} \mathbb{C}^{-1} \sigma \cdot \sigma
            \Big)
        \, dx
    \right)^{1/2} ,
\end{equation}
where $\mathbb{C} \in L(\mathbb{R}^{{n}\times {n}}_{\rm sym})$, $\mathbb{C}^T = \mathbb{C}$, $\mathbb{C} > 0$, is a nominal elasticity tensor.
We shall tacitly identify $\C$ with a tensor in $L(\mathbb{R}^n)$ by $\C\xi=\C(\xi+\xi^T)/2$.
The Data-Driven problem of linear elasticity is then characterized by the material data set
\begin{equation}\label{thiUw4}
    {\setD}
    =
    \{ (\epsilon,\sigma) \in Z \, : \ \sigma = \mathbb{C} \epsilon \} ,
\end{equation}
Fig.~\ref{3rlASw}a, and the constraint set
\begin{equation}\label{s6oAzi}
    {\setE} = \{ z \in {Z} : \  (\ref{sIes1A}) \text{ and } (\ref{fRoa1l}) \} .
\end{equation}
A straightforward calculation gives
\begin{equation}\label{Y6ebRi}
    d^2(z,{\setD})
    =
    \int_\Omega
        \frac{1}{4}
        \mathbb{C}^{-1}
        (\sigma - \mathbb{C} \epsilon)
        \cdot
        (\sigma - \mathbb{C} \epsilon)
    \, dx ,
\end{equation}
where for $z\in Z$ and $A\subseteq Z$ we write $d(z,A)=\inf\{\|z-a\|: a\in A\}$. In addition, by (\ref{eqpartintegrh12hm12}) on ${\setE}$ we have the power identity
\begin{equation}\label{GIe2oa}
    \int_\Omega \sigma(x) \cdot \epsilon(x) \, dx
    =
    \int_\Omega
        f(x) \cdot u(x)
    \, dx +
    \duality{\sigma\nu}{u}
\end{equation}
where $\nu$ is the outward unit normal to the boundary. We remark that
\begin{equation}
\begin{split}
    \duality{\sigma\nu}{u}
=   & \int_{\Gamma_D}
        \sigma(x)\nu(x) \cdot g(x)
    \, d\mathcal{H}^{n-1}(x)
   \\& +
    \int_{\Gamma_N}
        h(x) \cdot u(x)
    \, d\mathcal{H}^{n-1}(x) ,
    \end{split}
\end{equation}
for regular functions that obey (\ref{fRoa1l})-(\ref{sIes1A}).

The following lemma generalizes the classical Helmholtz-Hodge decomposition theorem to tensor fields (cf., \eg, \cite{GiraultRaviart1979,Temam1979BuchNavierStokes}). Here and subsequently, we write
\begin{equation}\label{Hlusp6}
    {e}(u) = \frac{1}{2}(\nabla u + \nabla u^T) ,
\end{equation}
for the strain operator on displacement vector fields $u \in H^1(\Omega; \mathbb{R}^{n})$.

\begin{lem}[Tensor Helmholtz decomposition]\label{B4Apia}
Let $\Omega \subset \mathbb{R}^{{n}}$ be open, bounded, Lipschitz. Let $\Gamma_D$, $\Gamma_N$ be disjoint open subsets of $\partial\Omega$ that obey (\ref{eqassGammaDN}).
Let %With $\mathbb{C} \in L(\mathbb{R}^{{n}\times {n}}_{\rm sym})$, $\mathbb{C}^T = \mathbb{C}$, $\mathbb{C} > 0$, let
\begin{subequations}
\begin{align}
    & \label{eqFRi9ji}
    M
    =
    \{
       {e}(u),
        \ u \in H^1(\Omega; \mathbb{R}^{n}),
        \ u = 0 \text{ on } \Gamma_D
    \} ,
    \\ & \label{tl9pHi}
    N
    =
    \{
        \sigma \in L^2(\Omega;\mathbb{R}^{{n}\times {n}}_{\rm sym}) ,
        \ {\rm div} \, \sigma = 0 ,
        \ \sigma \nu = 0 \text{ on } \Gamma_N
    \} .
\end{align}
\end{subequations}
Then, $M$ and $N$ are strongly (and weakly) closed in $L^2(\Omega,\mathbb{R}^{{n}\times {n}}_{\rm sym})$, $L^2(\Omega;\mathbb{R}^{{n}\times {n}}_{\rm sym}) = M \oplus N$ and the decomposition is orthogonal.% in the sense of (\ref{nIus4o}).
\end{lem}
We recall that $\sigma \nu = 0 \text{ on } \Gamma_N$ means that the function $\sigma\nu\in H^{-1/2}(\partial\Omega;\R^n)$ obeys $\duality{\sigma\nu}{g}=0$ for all $g\in H^{1/2}(\partial\Omega;\R^n)$ that vanish $\calH^{n-1}$-almost everywhere on $\partial\Omega\setminus\Gamma_N$ (this means, \ie on $\Gamma_D$).

\begin{proof}
The mapping ${\rm div} : L^2(\Omega;\mathbb{R}^{{n}\times {n}}_{\rm sym})$ $\to$ $H^{-1}(\Omega;\mathbb{R}^{n})$ is strongly continuous. Hence, $L^2_{\rm div}(\Omega,\mathbb{R}^{{n}\times {n}}_{\rm sym}) =$ preimage of $\{0\}$ by ${\rm div}$ is strongly closed in $L^2(\Omega;\mathbb{R}^{{n}\times {n}}_{\rm sym})$. The mapping $\sigma \mapsto \sigma \nu$ is strongly continuous from $L^2_{\rm div}(\Omega;\mathbb{R}^{{n}\times {n}}_{\rm sym})$ to $H^{-1/2}(\partial\Omega:\mathbb{R}^n)$. Therefore, if $\sigma_j\in N$ converges in $L^2$ to $\sigma$, then for any $g\in H^{1/2}(\partial\Omega;\R^n)$ with $g=0$ \ae on $\Gamma_D$, we obtain
\begin{equation}
    \duality{\sigma\nu}{g}=\lim_{j\to\infty}\duality{\sigma_j\nu}{g}=0.
\end{equation}
Therefore $N$ is closed. It is easily verified that $M$ is also closed.

Let $ {e}(u) \in M$ and $\sigma \in N$. Then, (\ref{eqpartintegrh12hm12})  gives
\begin{equation}
    \int_\Omega \sigma \cdot {e}(u) \, dx
    =
   \duality{\sigma\nu}{u}
    =
    0 ,
\end{equation}
which proves orthogonality.

Let $\sigma \in (M \oplus N)^\perp$. Then $\sigma \in M^\perp$ and $\sigma \in N^\perp$. For any $u^*\in C^\infty_c(\Omega,\R^n)$ we have $e(u^*)\in M$ and by the
 assumption $\sigma \in M^\perp$
 we obtain $\Div \sigma=0$ in $\Omega$. Therefore $\sigma\nu\in H^{-1/2}(\partial\Omega;\R^n)$ and
\begin{equation}
    0 =
    \int_\Omega \sigma \cdot {e}(u) \, dx
    =
    \duality{\sigma\nu}{u}
\end{equation}
for all $u \in H^1(\Omega; \mathbb{R}^{n})$, $u = 0$ on $\Gamma_D$. This means that $\sigma\nu=0$ on $\Gamma_N$ and therefore $\sigma \in N$. By the second inclusion, $\sigma \in N \cap N^\perp = \{0\}$. It thus follows that $L^2(\Omega;\mathbb{R}^{{n}\times {n}}_{\rm sym}) = M \oplus N$ and  $M = N^\perp$.
\end{proof}

\begin{thm}[Existence and uniqueness]\label{1Lublu}
Let $\Omega \subset \mathbb{R}^{{n}}$ be open, bounded, Lipschitz. Let $\Gamma_D$, $\Gamma_N$ be disjoint open subsets of $\partial\Omega$ that obey (\ref{eqassGammaDN}). Assume:

\begin{itemize}
\item[i)] $\mathbb{C} \in L(\mathbb{R}^{{n}\times {n}}_{\rm sym})$, $\mathbb{C}^T = \mathbb{C}$, $\mathbb{C} > 0$.
\item[ii)] $f \in L^2(\Omega;\mathbb{R}^{{n}})$, $g \in H^{1/2}(\partial\Omega;\R^n)$, $h \in H^{-1/2}(\partial\Omega;\R^n)$.
\end{itemize}
Let ${\setD}$ be as in (\ref{thiUw4}) and the constraint set ${\setE}$ as in (\ref{sIes1A}) and (\ref{fRoa1l}). Then, the Data-Driven problem
\begin{equation}\label{Pou8rO}
    \min\{ d(z,{\setD}),\ z \in {\setE} \}
\end{equation}
has a unique solution. Moreover, the Data-Driven solution satisfies
\begin{equation}\label{s7lEPr}
    \sigma = \mathbb{C} \epsilon .
\end{equation}
\end{thm}

\begin{proof}
By Lemma~\ref{B4Apia} and Hahn-Banach, the constraint set ${\setE}$ is weakly closed in $Z$. In addition, we note from (\ref{Y6ebRi}) that $d^2(\cdot,{\setD})$ is weakly lower-semicontinuous in ${Z}$, hence in ${\setE}$.
Let $(z_h) \subset {\setE}$, $d(z_h,{\setD}) \leq C < \infty$. From (\ref{Y6ebRi}),
\begin{equation}\label{youd3A}
    d^2(z_h,{\setD})
    =
    \int_\Omega
        \frac{1}{4} \,
        \mathbb{C}^{-1}
        (\sigma_h(x) - \mathbb{C} \epsilon_h(x))
        \cdot
        (\sigma_h(x) - \mathbb{C} \epsilon_h(x))
    \, dx .
\end{equation}
Expanding the square and integrating by parts using the power identity (\ref{GIe2oa}), we obtain
\begin{equation}\label{prlEj7}
\begin{split}
    &
    \| z_h \|^2
    =
    2d^2(z_h,{\setD})
    +
    \int_\Omega
        f(x) \cdot u_h(x)
    \, dx
    + \duality{\sigma\nu}{u}.
\end{split}
\end{equation}
An appeal to the Korn, Poincar\'e and trace theorems gives the estimates
\begin{equation}
 \|u_h\|_{H^1(\Omega;\R^n)} \le C(\|\epsilon_h\|_{L^2(\Omega;\R^{n\times n}_\sym)} + \|g\|_{H^{1/2}(\partial\Omega;\R^n)}),
\end{equation}
\begin{equation}
 \|u_h-g\|_{H^{1/2}(\partial\Omega;\R^n)} \le C(\|\epsilon_h\|_{L^2(\Omega;\R^{n\times n}_\sym)} + \|g\|_{H^{1/2}(\partial\Omega;\R^n)}),
\end{equation}
and
\begin{equation}
 \|\sigma_h\nu\|_{H^{-1/2}(\partial\Omega;\R^n)} \le C(\|\sigma_h\|_{L^2(\Omega;\R^{n\times n}_\sym)} + \|f\|_{L^{2}(\Omega;\R^n)}).
\end{equation}
Therefore, writing $\dualitysh{\sigma \nu}{u}=\dualitysh{\sigma\nu}{g}+\dualitysh{\sigma\nu}{u-g}$,
\begin{equation}\label{SpleD9}
\begin{split}
    \| z_h \|^2
    \leq&
    2d^2(z_h,{\setD})
    +
    C
    \| f \|_{L^{2}(\Omega;\mathbb{R}^{{n}})}
   ( \| \epsilon_h \|_{L^2(\Omega;\mathbb{R}^{{n}})}+ \|g\|_{H^{1/2}(\partial\Omega;\R^n)})
     \\ &
   + C
    \| g \|_{H^{1/2}(\partial\Omega;\mathbb{R}^{{n}})}
    (\| \sigma_h \|_{L^2(\Omega;\mathbb{R}^{{n}})}+\|f\|_{L^{2}(\Omega;\R^n)})\\ &
    +
    C
    \| h \|_{H^{-1/2}(\partial\Omega;\mathbb{R}^{{n}})}
    (\| \epsilon_h \|_{L^2(\Omega;\mathbb{R}^{{n}})}+ \|g\|_{H^{1/2}(\partial\Omega;\R^n)}) ,
\end{split}
\end{equation}
and
\begin{equation}\label{T8iuhL}
    \| z_h \|^2
    \leq
2    d^2(z_h,{\setD})
    +
    C \| z_h \| .
\end{equation}
Therefore, $\| z_h \|$ is bounded in ${Z}$ and $(z_h)$ has a weakly-convergent subsequence. The existence of Data-Driven solutions then follows from Tonelli's theorem \cite{Tonelli1921}.

Let $z = (\epsilon,\sigma)$ be a Data-Driven solution. By minimality,
\begin{equation}\label{eqminimialitysigsa}
    \int_\Omega
        \frac{1}{4} \,
        \mathbb{C}^{-1}
        (\sigma(x) - \mathbb{C} \epsilon(x))
        \cdot
        (\sigma'(x) - \mathbb{C} \epsilon'(x))
    \, dx
    =
    0 ,
\end{equation}
for all $z' = (\epsilon',\sigma') \in {\setE}_0=M\times N \equiv $ the constraint set for homogeneous data. The sets $M$ and $N$ were defined in  Lemma~\ref{B4Apia}.
Choosing $\epsilon'=0$ and $\sigma'\in N$, from Lemma~\ref{B4Apia} we obtain
\begin{equation}
 \C^{-1}\sigma-\epsilon\in N^\perp=M ,
\end{equation}
which means that there is $u\in H^1(\Omega;\R^n)$ such that
\begin{subequations}
\begin{align}
    & \label{4leplE}
    \mathbb{C}^{-1} \sigma-\epsilon
    =e(u) ,
    &
    \text{in } \Omega ,
    \\ &
    u(x) = 0 ,
    &
    \text{on } \Gamma_D .
\end{align}
\end{subequations}
 Choosing $\sigma'=0$ and $\epsilon'\in M$ in (\ref{eqminimialitysigsa}), we obtain
$\sigma-\C\epsilon\in M^\perp=N$. Since
$\sigma-\C\epsilon=\C e(u)$,
\begin{subequations}
\begin{align}
    &
 \Div(\sigma-\C\epsilon)= \Div(\C e(u))=0
    &
    \text{in } \Omega ,
    \\ &
    u = 0 ,
    &
    \text{on } \Gamma_D ,
    \\ &
    \mathbb{C} e(u) \cdot \nu= 0 ,
    &
    \text{on } \Gamma_N .
\end{align}
\end{subequations}
This implies
\begin{equation}
    \int_\Omega \C e(u) \cdot e(u) dx = \duality{\C e(u)\nu}{u}=0 ,
\end{equation}
which requires $u = 0$ and (\ref{s7lEPr}) follows from (\ref{4leplE}). Let now $z'$ and $z''$ be two Data-Driven solutions. Then, by linearity $z'-z''\in \setD\cap \setE_0$ is a Data-Driven solution with zero forcing. By identity (\ref{prlEj7}),
\begin{equation}
    \| z' - z'' \|^2
    =
    2d^2(z'-z'',{\setD})
    =
    0 ,
\end{equation}
and $z'=z''$ in ${Z}$.
\end{proof}

We note that, by (\ref{s7lEPr}), the unique Data-Driven solution coincides with the classical solution of linear elasticity. In particular, the Data-Driven solution does not change if the norm (\ref{nIus4o}) is replaced by an equivalent one, as for example the $L^2$ norm of $z$. We also note that the distance $d(\cdot,{\setD})$, eq.~(\ref{Y6ebRi}), does not control the norm of ${Z}$, which compounds the issue of coercivity. However, on the constraint set ${\setE}$ we have the power identity (\ref{GIe2oa}), which in turn yields (\ref{youd3A}) and the requisite coercivity of $d(\cdot,{\setD})$ on ${\setE}$.

\subsection{Approximation of the material data set}

Next, we consider a sequence $(\setD_h)$ of material data sets in $Z$, cf.~Fig.~\ref{3rlASw}b, and endeavor to ascertain conditions under which the solutions of the $\setD_h$-problems converge to solutions of a certain $\setD$-problem for a limiting material data set $\setD$. In order to set the general framework, let the phase space space $Z$ be a metric space and let $\setD \subset Z$ be a material data set and ${\setE} \in Z$ a constraint set. The corresponding Data-Driven problem then consists of finding
\begin{equation}\label{FroaW1}
    {\rm argmin} \, \{ d^2(z,\setD), \ z \in {\setE} \} ,
\end{equation}
or, equivalently,
\begin{equation}\label{sTieM2}
    {\rm argmin} \, \{ I_{\setE}(z) + d^2(z,\setD), \ z \in Z \} ,
\end{equation}
where
\begin{equation}
    I_\setE(x) =
    \left\{
        \begin{array}{ll}
            0, & \text{if } x\in \setE, \\
            \infty, & \text{otherwise} ,
        \end{array}
    \right.
\end{equation}
is the indicator function of $\setE \subset X$.

In order to consider approximation of material data sets we need different concepts of convergence of sets.
We specifically consider functionals with values in $$\overline \R = \R \cup\{  \infty\}.$$.

\def\calN{\mathcal{N}}
\begin{definition}[$\Gamma$-convergence]
Let $(X,\tau)$ be a topological space, $F_h:X  \to \overline \R$   a sequence of functionals. We say that $F: X \to  \overline \R$ is the $\Gamma(\tau)$-limit of $F_h$ if
\begin{equation}\label{eqdefgammac}
    F(x)
    =
    \sup_{U\in\calN(x)} \liminf_{h\to\infty} \inf_{y\in U} F_h(y)
    =
    \sup_{U\in\calN(x)} \limsup_{h\to\infty} \inf_{y\in U} F_h(y)
\end{equation}
for all $x\in X$, where $\calN(x)$ denotes the family of all open sets in $X$ that contain $x$.
\end{definition}

We recall  that every $\Gamma$-limit is lower semicontinuous with respect to $\tau$, see \cite{WD080}, Prop. 6.8.
We also remark that if $(X,\tau)$ is first countable $\Gamma$-convergence can be defined in terms of sequences.\footnote{We recall that a topological space is first countable if every point has a countable basis of neighbourhoods, which is in particular true for metric spaces.} Indeed, in this case (\ref{eqdefgammac}) is equivalent to the following two conditions
 (see \cite{WD080}, Prop. 8.1):
\begin{enumerate}
\item[i)] (liminf condition): For all $x\in X$ and all sequences $x_h\to x$,
\begin{equation}\label{eqdefGammaclb}
    F(x)\le \liminf_{h\to\infty} F_h(x_h) .
\end{equation}
\item[ii)] (limsup condition): For any $x\in X$ there is a sequence $x_h\to x$ such that
\begin{equation}\label{eqdefGammacub}
    \limsup_{h\to\infty} F_h(x_h)\le F(x) .
\end{equation}
\end{enumerate}

We recall that the weak topology of a reflexive, separable Banach space is metrizable on bounded sets. Indeed, if $T_i \in X^*$ is a countable dense set of elements of $X^*$,  then the metric defined
by $d(x,y) = \sum_{i=0}^\infty  2^{-i}  \| T_i (x-y)\|/ (1 + \| T_i(x-y)\|)$ metrizes the weak topology of $X$.
Thus  $\Gamma$ convergence with respect to the weak topology can be characterized by means of sequences
 if the functionals $(F_h)$ are equicoercive, see \cite{WD080}, Prop. 8.10. The same argument gives the following result.

\begin{lemma} \label{le:Gamma_sequence}
Let $(X,\tau)$ be a topological space and assume that there exists a metric  $d$ on $X$ such that on bounded sets the topology  $\tau_d$ defined by $d$ agrees with $\tau$. Assume that
\begin{equation}
    \hbox{$X_C=\{x\in X: F_h(x)<C \text{ for some } h\}$ is bounded for each $C \in \R$.}
\end{equation}
Then, conditions i) and ii)  are equivalent to $F = \Gamma_\tau-\lim_{h \to \infty} F_h$.
\end{lemma}

\begin{definition}[Mosco convergence of functions]
A sequence $(F_h)$ of functions from a Banach space $X$ to $\overline{\mathbb{R}}$
 converges to another function $F : X \to \overline{\mathbb{R}}$ in the sense of Mosco, or $F = M{-}\lim_{h\to\infty}F_h$, if
\begin{itemize}
\item[i)] For every sequence $(x_h)$ converging weakly to $x$ in $X$,
\begin{equation}
    \liminf_{h\to\infty} F_h(x_h) \geq F(x) .
\end{equation}
\item[ii)] For every $x \in X$, there is a sequence $(x_h)$ converging strongly to $x$ in $X$ such that
\begin{equation}
    \lim_{h\to\infty} F_h(x_h) = F(x) .
\end{equation}
\end{itemize}
% \hfill$\square$
\end{definition}

\begin{definition}[Mosco convergence of sets]
A sequence $(\setE_h)$ of subsets of a Banach space $X$ converges to $\setE \subset X$ in the sense of Mosco, or $\setE = M{-}\lim_{h\to\infty}\setE_h$, if
\begin{equation}
    I_\setE = M{-}\lim_{h\to\infty} I_{\setE_h} .
\end{equation}
% \hfill$\square$
\end{definition}

\begin{lemma} \label{le:mosco_swlsc}
Every Mosco limit functional $F$ is weakly sequentially lower semicontinuous. In particular every Mosco limit set $\setE$ is weakly sequentially closed and $\setE =  M{-}\lim_{h\to\infty}\setE$ if and only if $\setE$ is weakly  sequentially closed.
\end{lemma}

\begin{proof}
If $F$ is a Mosco limit,  $x_k \weakto x$ and $L = \liminf_{k \to \infty} F(x_k)$ we need to show that $F(x) \le L$. Passing to a subsequence we may assume that $L =  \lim_{k \to \infty} F(x_k)$. By the $\limsup$ property there exist $x_{k,h}$ such that $\| x_{k,h} - x_k\| \le 2^{-h}$ and $\lim_{h \to \infty} d_{\overline \R}(F_h(x_{k,h}), F(x_k)) \le 2^{-h}$. Thus there exists an increasing sequence $h \mapsto k_h$ such that $y_h = x_{k_h, h}$ satisfies $\| y_h - x_{k_h}\| \le 2^{-h}$ and $d_{\overline \R}(F(y_h), L) \le 2^{-h+1}$. Hence $y_h \weakto x$ and $F(y_h) \to L$. Now the $\liminf$ property of Mosco convergence yields $F(x) \le L$. The assertions for sets follow by taking $F_h = I_{\setE_h}$.
\end{proof}

We have the following relation between $M$-convergence of material data sets and $\Gamma$-convergence of the corresponding Data-Driven functionals.

\begin{thm}\label{sTLuz7}
Let ${Z}$ be a reflexive, separable Banach space, $\setD$ and $(\setD_h)$ subsets of $Z$, ${\setE}$ a weakly  sequentially closed subset of $Z$. Suppose:
\begin{itemize}
\item[i)] (Mosco convergence) $\setD = M{-}\lim_{h\to\infty} \setD_h$ in $Z$.
\item[ii)] (Equi-transversality) There are constants $c > 0$ and $b \geq 0$ such that, for all $y \in \setD_h$ and $z \in {\setE}$,
\begin{equation}\label{m4ufrO}
%    d^2(y,z) \geq c (\|y\|^2 + \|z\|^2) - b .
\|y-z\| \geq c (\|y\| + \|z\|) - b .
\end{equation}
\end{itemize}
Then,
\begin{equation}\label{sPi5Sl}
    I_{\setE}(\cdot) + d^2(\cdot,\setD)
    =
    \Gamma{-}\lim_{h\to\infty}
    \Big( I_{\setE}(\cdot) + d^2(\cdot,\setD_h) \Big),
\end{equation}
with respect to the weak topology of $Z$.

If $(z_h)$ is a sequence of elements of $Z$ with $\sup_h I_\setE(z_h)+d^2(z_h,\setD_h)<\infty$
then there is a subsequence converging weakly to some $z\in \setE$.
\end{thm}

\begin{proof}
a) {\sl Compactness}.
Let $F_h(z) =  I_\setE(z_h)+d^2(z_h,\setD_h)$. Then
\begin{equation}
    \sqrt{F_h(z)}
    \ge
    I_\setE(z) + \inf_{y \in \setD_h} \| z - y\|
    \ge
    \inf_{y \in \setD_h} c(\|y\| + \|z\|) -b \ge c \|z\| - b.
\end{equation}
Hence, by Lemma \ref{le:Gamma_sequence} $\Gamma$ convergence with respect to the weak topology  of $F_h$ is characterized by (\ref{eqdefGammaclb})-(\ref{eqdefGammacub}) (for weakly convergent sequences).

b) {\sl Limsup inequality}.
Let $z \in Z$ (otherwise there is nothing to show). By Lemma  \ref{le:mosco_swlsc} $\setD$ is weakly sequentially closed.
Thus there exist $y \in \setD$ such that $\| z - y\| = d(z, \setD)$. By the limsup inequality in Mosco convergence of $\setD_h$
there is a sequence $y_h\in \setD_h$ such that $y_h\to y$ strongly, which implies
\begin{equation}
\limsup_{h\to\infty}  d(z,\setD_h)\le
\limsup_{h\to\infty} \|z-y_h\| \le
\|z-y\| = d(z,\setD).
\end{equation}
Therefore the constant sequence gives the limsup inequality (\ref{eqdefGammacub}).

c) {\sl Liminf inequality}. Let $z_h \rightharpoonup z$ in $Z$.
Let
$$ L := \liminf_{h \to \infty} F_h(z_h) = \liminf_{h \to \infty} d^2(z_h, \setD_h).$$
We need to show that $F(z) \le L$.  We may assume that $L < \infty$ and passing to a subsequence (not renamed) we may assume that
$z_h \in \setE$ and  $L = \lim_{h \to \infty} d^2(z_h, \setD_h)$. There exist $y_h \in D_h$ such that
$$ L = \lim_{h \to \infty} \| z_h - y_h\|^2.$$
It follows from equi-transversality that the sequence $y_h$ is bounded. Passing to a subsequence (not renamed) we may assume that $y_h \weakto y$. By the liminf inequality in Mosco convergence of $D_h$ we get
\begin{equation}
    0
    =
    \liminf_{h\to\infty} I_{\setD_h}(y_h)
    \geq
    I_\setD(y)
    \geq
    0 .
\end{equation}
Hence, $I_\setD(y) = 0$ and $y \in \setD$. In addition, $z_h-y_h \rightharpoonup z-y$. By convexity,
\begin{equation}
L    =   \liminf_{h\to\infty} \|z_h-y_h\|^2
    \geq
\|z-y\|^2   \geq
    d^2(z,\setD) = F(z)
\end{equation}
and  (\ref{eqdefGammaclb}) is proved.
\end{proof}

We note that, by Lemma  \ref{le:mosco_swlsc}
the limiting material data set $\setD$ must be necessarily weakly
sequentially closed in $Z$. All strongly closed convex sets, including closed
linear subspaces of $Z$, satisfy this requirement.

\begin{figure} [ht]
\centering
\includegraphics[width=0.75\linewidth]{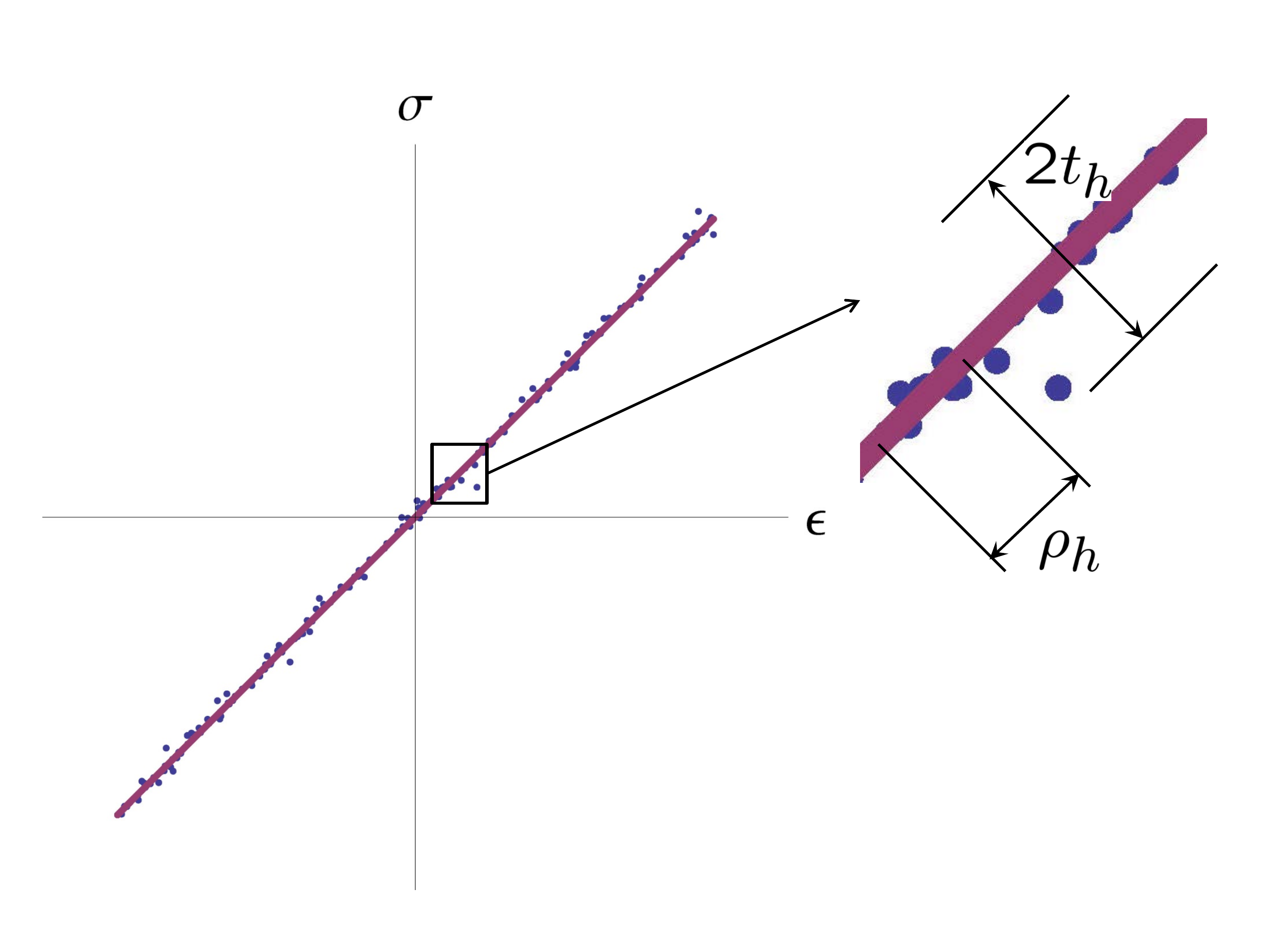}
\caption{Schematic of convergent sequence of material data sets. The parameter $t_h$ controls the spread of the material data sets away from the limiting data set and the parameter $\rho_h$ controls the density of material data point.}
\label{c7aKLe}
\end{figure}

Conditions for the existence of solutions of Data-Driven problems with weakly
sequentially  closed material data sets follow directly from Theorem~\ref{sTLuz7}.

\begin{cor}    \label{s4leBr}
Let ${Z}$ be a reflexive, separable Banach space, $\setD$ and  ${\setE}$ be weakly sequentially closed subsets of $Z$.
Suppose the transversality assumption (\ref{m4ufrO}) holds for any $y\in \setD$, $z\in \setE$. Then, the Data-Driven problem (\ref{sTieM2}) has solutions.
\end{cor}

\begin{proof}
By Lemma  \ref{le:mosco_swlsc},
 Theorem~\ref{sTLuz7} can be applied to the constant sequence $\setD_h=\setD$, with the result that $I_{\setE}(\cdot) + d^2(\cdot,\setD)$ is weakly lower semicontinuous and coercive. Existence of solutions then follows from Tonelli's theorem \cite{Tonelli1921}.
\end{proof}

We now recover the existence Theorem~\ref{1Lublu} for the Data-Driven linear elastic problem as a special case of Theorem~\ref{sTLuz7} and Corollary~\ref{s4leBr}.

\begin{cor}\label{corlinearleastc}
Under the assumptions of Theorem~\ref{1Lublu}, the linear-elastic Data-Driven problem (\ref{Pou8rO}) has solutions.
\end{cor}

\begin{proof}
It follows immediately that $\setD = {M}{-}\lim_{h\to\infty} \setD$. In addition, the intersection of $\setD$ and $\setE$ consists of one single point corresponding to the unique solution $z_0$ of the linear elasticity problem. Let $\setE_0=M\times N$ be the constraint set for null data, \ie, (\ref{s6oAzi}) with $f=0$, $g=0$ and $h=0$. We then have $\setE = z_0 + \setE_0$ for some $z_0\in \setE$. We claim that there is $c > 0$ such that
\begin{equation}\label{joa9lA}
    \| y' - z' \| \geq c ( \| y' \| + \| z' \| )
\end{equation}
or every $y' \in \setD$ and $z' \in \setE_0$.
To see this, we write $z'=(\epsilon',\sigma')$ and by (\ref{Y6ebRi}) compute
\begin{equation}
    \|y'-z'\|^2\geq d^2(z',{\setD})
    =
    \int_\Omega
        \frac{1}{4}
        \mathbb{C}^{-1}
        (\sigma' - \mathbb{C} \epsilon')
        \cdot
        (\sigma' - \mathbb{C} \epsilon')
    \, dx .
\end{equation}
By Lemma \ref{B4Apia} the fields $\epsilon'\in M$ and $\sigma'\in N$ are orthogonal, therefore
\begin{equation}
    \|y'-z'\|^2\geq d^2(z',{\setD})
    =
    \int_\Omega
        \frac{1}{4}
        \mathbb{C}^{-1}\sigma'\cdot\sigma'
        +\frac{1}{4}
        \mathbb{C}\epsilon'\cdot\epsilon'
    \, dx =\frac12 \|z'\|^2.
\end{equation}
Finally,
\begin{equation}
  \|y'-z'\|\geq
  \frac12 \|y'-z'\|+
\frac12   d(z',{\setD})
\geq
  \frac12 \|y'-z'\|+
\frac 1{2\sqrt2}   \|z'\|
\geq c ( \|y'\|+\|z'\|)
\end{equation}
proves (\ref{joa9lA}).
%Suppose otherwise. Then, there are sequences $(y'_h)$ and $(z'_h)$ such that $\| y'_h \| = 1$, $\| z'_h \| = 1$ and $\| y'_h - z'_h \| \to 0$. By Banach-Alaoglu, there are $y' \in Z$, $z' \in Z$ and subsequences, not renamed, such that $y'_h \rightharpoonup y'$ and $z'_h \rightharpoonup z'$. By the weak closedness of $D$ and $E_0$, it follows that $y' \in D$ and $z' \in E_0$. We additionally have $y'_h - z'_h \rightharpoonup y' - z'$. By lower-semicontinuity, $\| y' - z' \| \leq \liminf_{h\to z} \| y'_h - z'_h \| = 0$, hence $y' = z'$. Since $D \cap E_0 = \{ 0 \}$, it follows that $y' = z' = 0$, which results in a contradiction.
Let now $y\in \setD$, $z\in \setE$, and define $z'=z-z_0\in\setE_0$. Then
a triangular inequality gives $   \| z' \| \geq \| z \| - \| z_0 \|$ and
we obtain
\begin{equation}
    \| y - z \|
    \geq
    c ( \| y \| + \| z \| )
    -
     c \| z_0 \| ,
\end{equation}
which is (\ref{m4ufrO}).
\end{proof}

A particular case of interest is the following, cf.~Fig.~\ref{c7aKLe}. The case concerns approximation by sequences of material data sets that converge to a limiting material in a certain uniform sense. In practice, such approximations may result from {\sl sampling} the material behavior experimentally or through multiscale computational models.

\begin{lemma}\label{0rIuth}
Let $Z$ and $\setE$ be as in (\ref{eqdefZlinearelast}) and (\ref{s6oAzi}). Suppose that
\begin{equation}
    \setD_h = \{ z \in Z \, : \ z(x) \in \setD_{{\rm loc},h} \text{ \ae~in  } \Omega \} ,
\end{equation}
for some sequence of local material data sets $\setD_{{\rm loc},h} \subset \mathbb{R}^{{n} \times {n}}_{\rm sym} \times \mathbb{R}^{{n} \times {n}}_{\rm sym}$. Let
\begin{equation}
    \setD = \{ z \in Z \, : \ z(x) \in \setD_{{\rm loc}} \text{ \ae~in } \Omega \} ,
\end{equation}
where
\begin{equation}
    \setD_{\rm loc} = \{ (\epsilon,\sigma) \in \mathbb{R}^{{n} \times {n}}_{\rm sym} \times \mathbb{R}^{{n} \times {n}}_{\rm sym}, \ \sigma = \mathbb{C} \epsilon \} .
\end{equation}
Assume that:
\begin{itemize}
\item[i)] (Fine approximation)  There is a sequence $\rho_h \downarrow 0$ such that
\begin{equation}
    d(\xi, \setD_{{\rm loc},h}) \leq \rho_h, \qquad \forall \xi \in \setD_{\rm loc} .
\end{equation}
\item[ii)] (Uniform approximation) There is a sequence $t_h \downarrow 0$ such that
\begin{equation}
    d(\xi, \setD_{\rm loc}) \leq t_h, \qquad \forall \xi \in \setD_{{\rm loc},h} .
\end{equation}
\end{itemize}
Then, $\setD = M{-}\lim_{h\to\infty} \setD_h$ in $Z$.
\end{lemma}

\begin{proof}
%Compactness follows as in Corollary \ref{corlinearleastc}, hence we work with sequences.
a) {\sl Liminf inequality}.
Let $z_h \rightharpoonup z$ in $Z$. By passing to a subsequence, not renamed, we may assume $z_h \in \setD_h$. Then, by (ii) we have $d(z_h,\setD) \to 0$ and, by the lower-semicontinuity of $d(\cdot,\setD)$,
\begin{equation}
    0 = \liminf_{h \to \infty} d(z_h,\setD) \geq d(z,\setD).
\end{equation}
Hence $z\in \setD$ and
\begin{equation}
    0 = \liminf_{h\to 0} I_{\setD_h}(z_h) = I_{\setD}(z) = 0 .
\end{equation}

b) {\sl Limsup inequality}.
Let $z \in \setD$. By (i) we can choose a measurable function $z_h$ such that $z_h(x)\in \setD_{{\rm loc},h}\cap \overline B_{\rho_h}(z(x))$. Then, $z_h \to z$ and
\begin{equation}
    I_{\setD}(z) = 0 = \lim_{h\to 0} I_{\setD_h}(z_h),
\end{equation}
as required.
\end{proof}

A noteworthy consequence of Theorem~\ref{sTLuz7} is that, when the material set $\setD$ is weakly  sequentially closed, the convergence of approximating material-set sequences $(\setD_h)$ can be elucidated independently of the constraint set ${\setE}$. This situation is fundamentally different when the material set $\setD$ fails to be weakly closed, as shown next.

\section{Relaxation of the two-well problem in elasticity}
\label{secrelax}

\begin{figure}[ht]
	\begin{subfigure}{0.45\textwidth}\caption{} \includegraphics[width=0.99\linewidth]{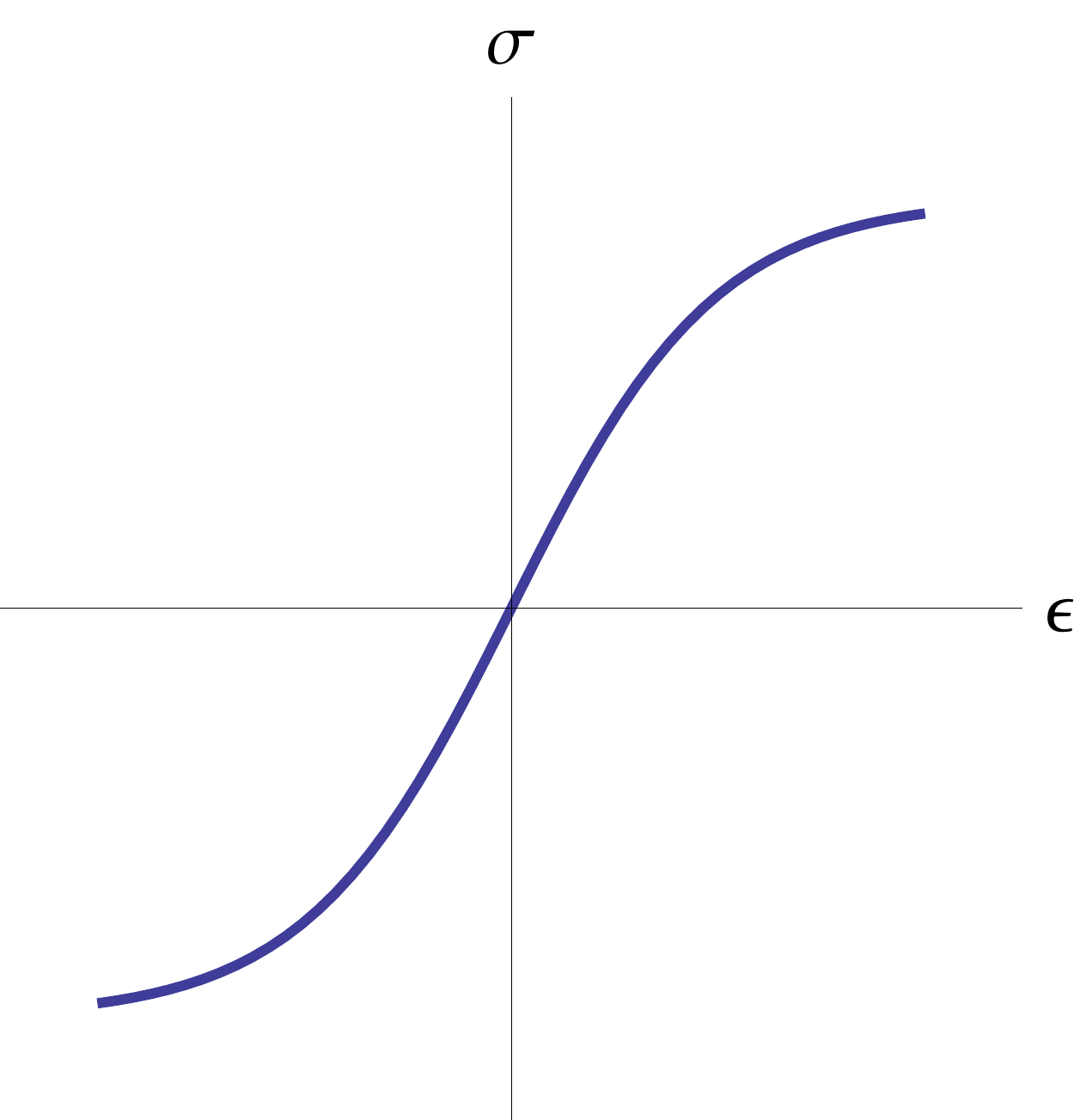}
	\end{subfigure}
    $\quad$
	\begin{subfigure}{0.45\textwidth}\caption{} \includegraphics[width=0.99\linewidth]{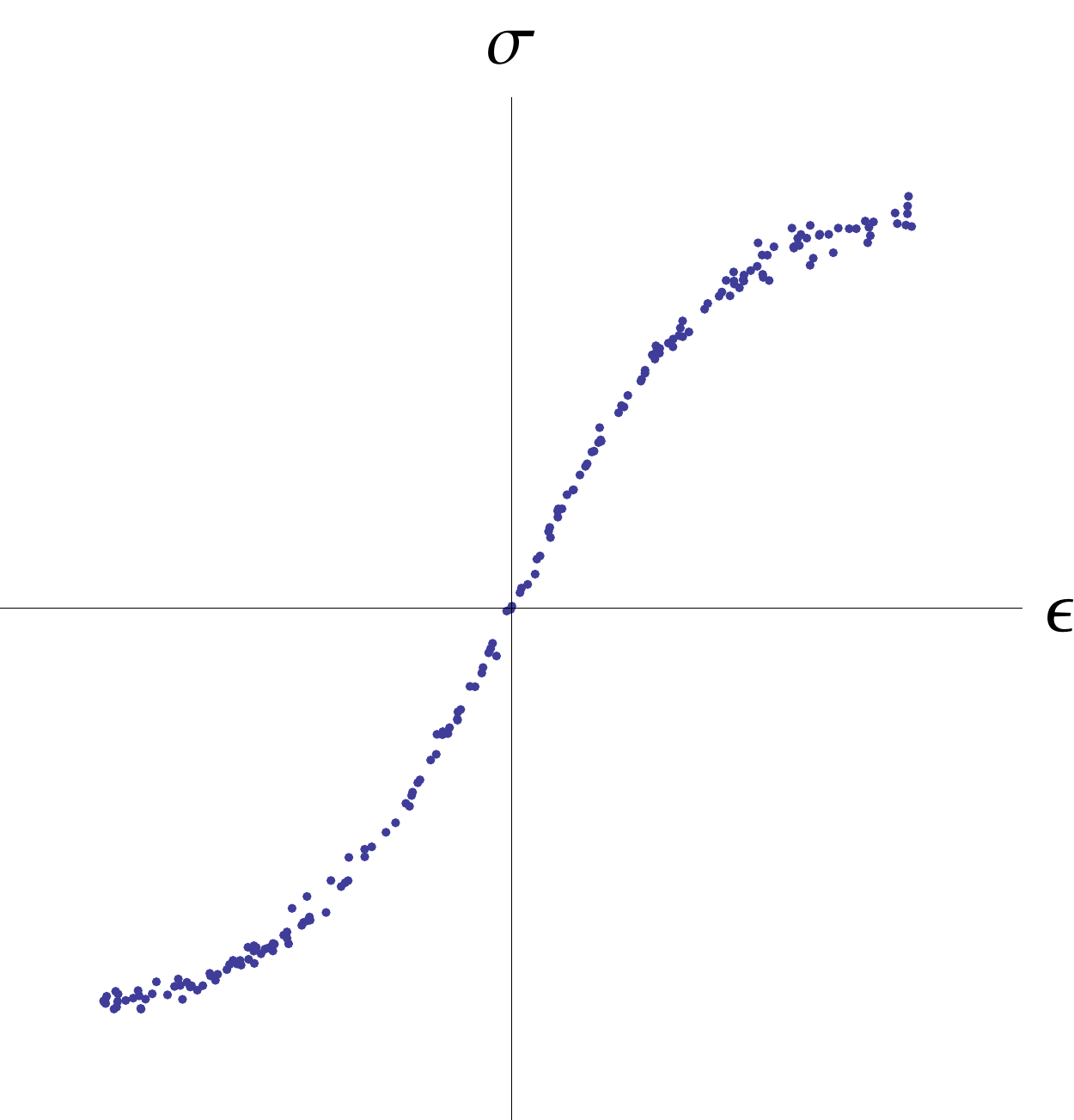}
	\end{subfigure}
	\caption{a) Local material data set for nonlinear elasticity. b) Sampled local material data set for nonlinear elasticity.} \label{4iAcro}
\end{figure}

Next, we turn attention to non-convex material sets $\setD$ failing to be weakly closed, cf.~fig.~\ref{4iAcro}. We first discuss abstractly the Data-Driven approximation and relaxation concepts, assuming that $Z$ is a reflexive separable Banach space and $\setE$ a weakly-closed subset. The corresponding Data-Driven problem is again (\ref{FroaW1}) or, equivalently, (\ref{sTieM2}). Furthermore, since
\begin{equation}
    d(z,\setD) = \inf_{y \in Z} \Big( I_\setD(y) + \|y-z\| \Big) ,
\end{equation}
an alternative definition of Data-Driven problem (\ref{sTieM2}) is
\begin{equation}\label{si4xOu}
    {\rm argmin} \,
    \{ F(y,z), \ (y,z) \in Z\times Z \} ,
\end{equation}
with
\begin{equation}\label{d3upiA}
    F(y,z) = I_\setD(y) + I_{\setE}(z) + \|y-z\|^2= I_{\setD\times {\setE}}(y,z) + \|y-z\|^2.
\end{equation}
The choice of the exponent 2 is only for definiteness and does not influence any statement.
%We wish to ascertain the properties of this class of functionals, including existence of minimizers, relaxation and convergence with respect to the data set.
In Section \ref{seconed} and \ref{sectwowell} we shall then turn to specific examples of multistable materials, restricting attention to linearized kinematics.
%, so that the phase space ${Z}$ can be identified
%with $L^2(\Omega;\mathbb{R}^{{n} \times {n}}_{\rm sym}) \times L^2(\Omega;\mathbb{R}^{{n}\times {n}}_{\rm sym})$, metrized by norm (\ref{nIus4o}). The constraint set ${\setE}$, Eq.~(\ref{s6oAzi}) will then consist
%of the elements of $Z$ that are compatible and in equilibrium.

\subsection{Approximation of the material data set}

In addition to its strong $(S)$ and weak $(W)$ topologies, with corresponding convergence of sequences denoted $\to$ and $\rightharpoonup$, respectively, we endow $Z\times Z$ with the following intermediate topology $(\topoD)$. The abstract definitions can be given for any reflexive, separable Banach space $Z$.

\begin{definition}[Data convergence]
A sequence $(y_h,z_h)$ in $Z\times Z$ is said to converge to $(y,z) \in Z\times Z$ in the Data topology, denoted $(y,z) = {\topoD}{-}\lim_{h\to\infty} (y_h,z_h)$, if $y_h \rightharpoonup y$, $z_h \rightharpoonup z$ and $y_h-z_h \to y-z$.
\end{definition}

We denote by $\Gamma(\tau){-}\lim_{h\to\infty}$, $\tau = W, \topoD, S$, the $\Gamma$-limit of sequences of functions over $Z\times Z$ and by $K(\tau){-}\lim_{h\to\infty}$, $\tau = W, \topoD, S$, the Kuratowski limit of sequences of sets in $Z\times Z$. We recall that Kuratowski convergence of sets is equivalent to $\Gamma$ convergence of the indicator functions.

\newcommand\calO{\mathcal{O}}
\newcommand\calP{\mathcal{P}}
We remark that the Data convergence can be easily reformulated in terms of the weak topology. Indeed, letting $\calO_S, \calO_W\subset\calP(Z)$ be the families of open subsets with respect with the strong and the weak topology of $Z$, respectively, one considers on $Z\times Z$ the coarsest topology for which the sets $(A\times A') \cap \{(z,z'): z-z'\in C\}$ are open, for all $A,A'\in \calO_W$, $C\in\calO_S$.
In particular, there exists a metric $d$ such that the Data topology on bounded subsets of $Z \times Z$ agrees with the topology $\tau_d$ induced by $d$ and thus Lemma \ref{le:Gamma_sequence} implies that for coercive functionals $\Gamma$ convergence in the Data topology is given by the sequential characterization  \eqref{eqdefGammaclb}--\eqref{eqdefGammacub}.

The following theorem establishes a connection between convergence of material data sets and convergence of solutions of the associated Data-Driven problems.

\begin{thm}\label{wO2cri}
Let $\setD$ and $(\setD_h)$ be subsets of a reflexive separable Banach space $Z$, ${\setE}$ a weakly sequentially closed subset of $Z$. For $(y,z) \in Z\times Z$, let
\begin{equation}
    F_h(y,z)
    =
    I_{\setD_h}(y) + I_{\setE}(z) +  \|y-z\|^2
    =
    I_{\setD_h\times {\setE}}(y,z) + \|y-z\|^2 .
\end{equation}
Suppose:
\begin{itemize}
\item[i)] (Data convergence) $\setD\times {\setE} = K(\topoD){-}\lim_{h\to\infty} (\setD_h\times {\setE})$.
\item[ii)] (Equi-transversality) There are constants $c > 0$ and $b \geq 0$ such that, for all $y \in \setD_h$ and $z \in {\setE}$,
\begin{equation}\label{xlUd1o}
%    d^2(y,z) \geq c (\|y\|^2 + \|z\|^2) - b .
   \|y-z\| \geq c (\|y\| + \|z\|) - b .
\end{equation}
\end{itemize}
Then:
\begin{itemize}
\item[a)] If $F_h(y_h,z_h) \to 0$, there exists $z \in \setD \cap {\setE}$ such that,  up to subsequences, $(z,z)$ $=$ ${\topoD}{-}\lim_{h\to\infty} (y_h,z_h)$.
\item[b)] If $z \in \setD \cap {\setE}$, there exist a sequence $(y_h,z_h)$ in $Z\times Z$ such that $(z,z) = {\topoD}{-}\lim_{h\to\infty} (y_h,z_h)$ and $F_h(y_h,z_h) \to 0$.
\end{itemize}
\end{thm}

\begin{proof}
We first observe that, as in the proof of Theorem \ref{sTLuz7}, equi-transversality implies equi-coercivity, so that we can work with the sequential definition of $\Gamma$-convergence in (\ref{eqdefGammaclb})-(\ref{eqdefGammacub}).

a) Since $F_h(y_h,z_h) \to 0$, it follows that $y_h \in \setD_h$, $z_h \in {\setE}$, $ \|y_h-z_h\| \to 0$. By (ii), $(y_h)$ and $(z_h)$ are bounded. Therefore, there are $y \in Z$ and $z \in Z$ such that $y_h \rightharpoonup y$ and $z_h \rightharpoonup z$ up to subsequences. By the weak closedness of ${\setE}$, $z \in {\setE}$. By weak lower-semicontinuity,
\begin{equation}
    0 \leq  \|y-z\| \leq \liminf_{h\to\infty}  \|y_h-z_h\| = 0 .
\end{equation}
Hence $y = z$ and $(y,z) = {\topoD}{-}\lim_{h\to\infty} (y_h,z_h)$. By (i),
\begin{equation}
    0
    \leq
    I_\setD(y)
    =
    I_\setD(y) + I_{\setE}(z)
    \leq
    \liminf_{h\to\infty}
    \Big( I_{\setD_h}(y_h) + I_{\setE}(z_h) \Big)
    =
    0 ,
\end{equation}
hence $y \in \setD$.

b) Let $z \in \setD\cap {\setE}$. Then, by (i) there exists a sequence $(y_h,z_h)\in\setD_h\times\setE$ with limit $(z,z) = {\topoD}{-}\lim_{h\to\infty} (y_h,z_h)$.
In particular, we have $y_h - z_h \to z - z = 0$. Hence, by continuity of the norm,
\begin{equation}
    \lim_{h\to\infty} F_h(y_h,z_h)
    =
    \lim_{h\to\infty} \Big( I_{\setD_h}(y_h) + I_{\setE}(z_h) +  \|y_h-z_h\|^2 \Big)
    =
    0 ,
\end{equation}
as required.
\end{proof}

\begin{figure} [ht]
\centering
\includegraphics[width=0.75\linewidth]{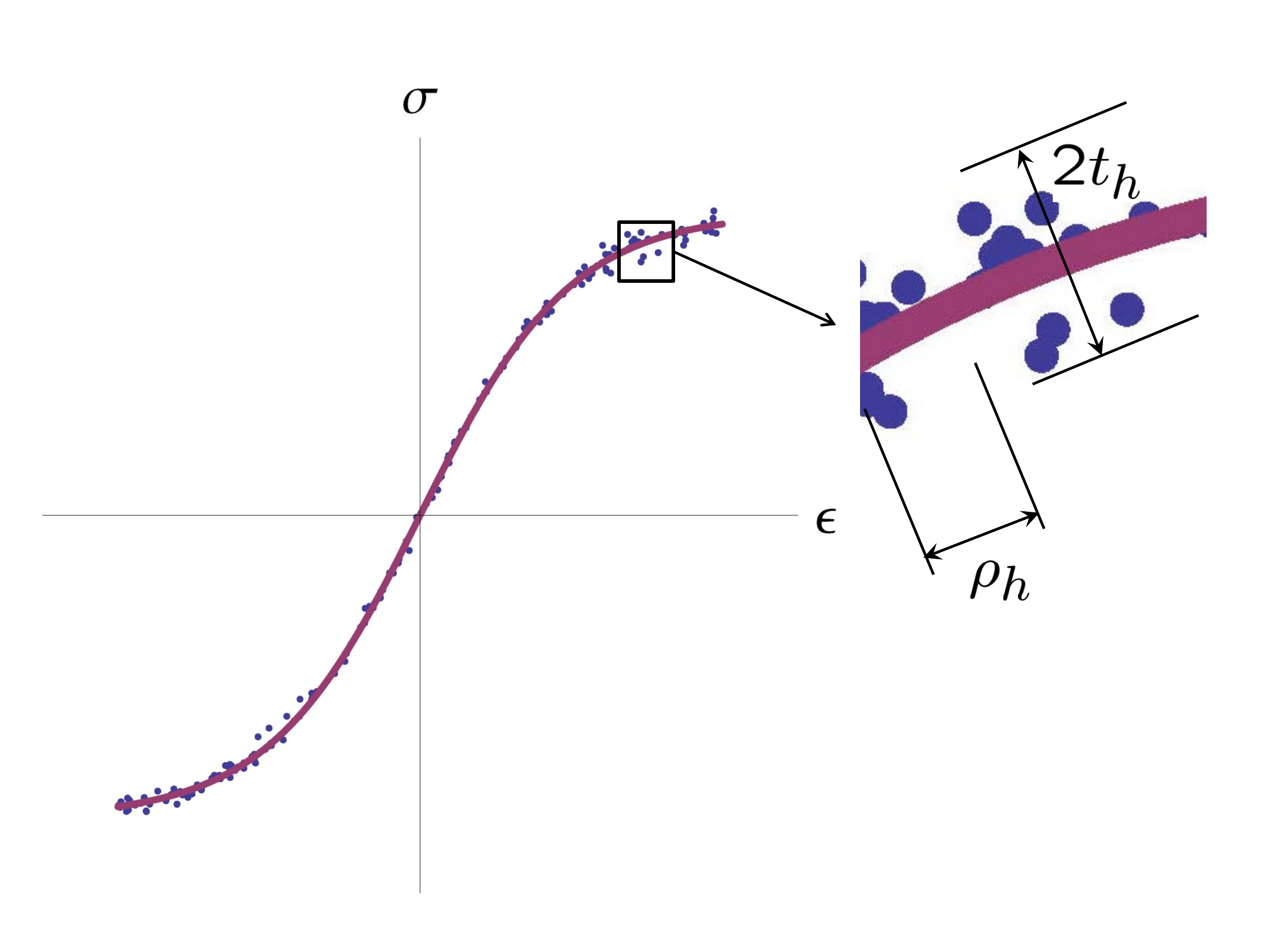}
\caption{Schematic of convergent sequence of material data sets. The parameter $t_h$ controls the spread of the material data sets away from the limiting data set and the parameter $\rho_h$ controls the density of material data point.} \label{0hiunI}
\end{figure}

We also show that Data relaxation is stable with respect to fine and uniform approximation of the material data sets.

\begin{thm}\label{fLew0a}
Let ${\setE} \subset Z$ be weakly  sequentially closed, $\setD$, $\bar\setD\subset Z$. Suppose:
\begin{itemize}
\item[i)] (Data convergence) $\bar{\setD}\times {\setE} = K(\topoD){-}\lim_{h\to\infty} (\setD\times {\setE})$.
\item[ii)] (Fine approximation) There is a sequence $\rho_h \downarrow 0$ such that
\begin{equation}
    d(\xi, \setD_{h}) < \rho_h, \qquad \forall \xi \in \setD.
\end{equation}
\item[iii)] (Uniform approximation) There is a sequence $t_h \downarrow 0$ such that
\begin{equation}
    d(\xi, \setD) < t_h, \qquad \forall \xi \in \setD_{h} .
\end{equation}
\item[iv)] (Transversality) There are constants $c > 0$ and $b \geq 0$ such that, for all $y \in \setD$ and $z \in {\setE}$,
\begin{equation}
    \|y-z\| \geq c (\|y\| + \|z\|) - b .
\end{equation}
\end{itemize}
Then, $\bar{\setD}\times {\setE} = K(\topoD){-}\lim_{h\to\infty} (\setD_h\times {\setE})$.
\end{thm}
\begin{remark}\label{remarklocalapprox}
If the data sets are local, in the sense that %We remark that (ii) and (iii) can be easily satisfied if the sets are local, in the sense that
\begin{equation}
    \setD_h = \{ z \in Z \, : \ z(x) \in \setD_{{\rm loc},h} \text{ a.~e.~in } \Omega \} ,
\end{equation}
for some sequence of local material data sets $\setD_{{\rm loc},h} \subset \mathbb{R}^{{n} \times {n}}_{\rm sym} \times \mathbb{R}^{{n} \times {n}}_{\rm sym}$, and analogously for $\setD$ and  $\setD_{{\rm loc}} \subset \mathbb{R}^{{n} \times {n}}_{\rm sym} \times \mathbb{R}^{{n} \times {n}}_{\rm sym}$, then the approximation properties (ii) and (iii) can be written as
\begin{equation}
    d(\xi, \setD_{{\rm loc},h}) \leq \rho_h, \qquad \forall \xi \in \setD_{\rm loc}
\end{equation}
and
\begin{equation}
    d(\xi, \setD_{\rm loc}) \leq t_h, \qquad \forall \xi \in \setD_{{\rm loc},h} ,
\end{equation}
respectively.
\end{remark}
We observe that the transversality condition iv) is only used here to ensure
the sequential characterization of the $K(\Delta)$ limit.

\begin{proof}
a) {\sl Coercivity}.
We first show that the equi-transversality (condition (ii) in Theorem \ref{wO2cri}) holds and the assertions of Theorem \ref{wO2cri} follow. To see this, let $y\in\setD_h$. By the uniform approximation property (iii) there is $\hat y\in \setD$ with $\|y-\hat y\|< t_h$ and therefore for all $z\in\setE$
\begin{equation}
\begin{split}
    \|y-z\| \geq  & \|\hat y-z\| - \|y-\hat y\|\ge  c (\|\hat y\| + \|z\|) - b -t_h \\
    \ge & c (\| y\| + \|z\|) - b -(1+c)t_h.
    \end{split}
\end{equation}
The sequence $t_h$ is obviously bounded, and therefore equi-transversality holds with $b'=b+(1+c)\max_h t_h$.
In particular, we can work with the definition of $\Gamma$ convergence by sequences.

b) {\sl Liminf inequality}.
Suppose that $(y,z) = \topoD-\lim_{h\to\infty} (y_h,z_h)$ in $Z\times Z$. We need to prove
\begin{equation}\label{eqlimiinfneedtopf}
    I_{\bar{\setD}\times \setE}(y,z) \le  \liminf_{h\to\infty}   I_{\setD_h\times \setE}(y_h,z_h).
\end{equation}
If $\liminf_{h\to\infty}   I_{\setD_h\times \setE}(y_h,z_h)=\infty$, then (\ref{eqlimiinfneedtopf}) holds. Otherwise, possibly passing to a subsequence, we can assume $y_h\in \setD_h$
and $z_h \in \setE$ for all $h$. By (iii) there are $\hat y_h\in \setD$ such that $\|\hat y_h-y_h\|\le t_h$.
In particular, $\hat y_h\weakto y$ and $\hat y_h-z_h\to y-z$ strongly, so that $(y,z) = \topoD-\lim_{h\to\infty} (\hat y_h,z_h)$. By (i)
\begin{equation}
    I_{\bar{\setD}\times \setE}(y,z) \le  \liminf_{h\to\infty}   I_{\setD\times \setE}(\hat y_h,z_h)=0.
\end{equation}
This concludes the proof of the lower bound.

c) {\sl Limsup inequality}. Let $(y,z) \in Z\times Z$.
We need to construct a sequence $(y_h, z_h)\in \setD_h\times \setE$ such that $(y,z) = \topoD-\lim_{k\to\infty} (y_h,z_h)$  in $Z\times Z$ and
\begin{equation}
   \limsup_{h\to\infty}   I_{\setD_h\times \setE}(y_h,z_h) \le I_{\bar{\setD}\times \setE}(y,z) .
\end{equation}
If $I_{\bar{\setD}\times \setE}(y,z) =\infty$ a constant sequence will do, therefore we can assume $y\in \bar\setD$, $z\in \setE$. By (i), there is a sequence $(\hat y_h,z_h)\in \setD\times\setE$ such that $(y,z) = \topoD-\lim_{h\to\infty} (\hat y_h,z_h)$  in $Z\times Z$. By (ii), for every $\hat y_h$ there is $y_h\in \setD_h$ such that $\|\hat y_h-y_h\|\le \rho_h$. This implies  $(y,z) = \topoD-\lim_{k\to\infty} (y_h,z_h)$. Furthermore,
\begin{equation}
    \lim_{h\to\infty} I_{\setD_h\times \setE}(y_{h}, z_{h}) = I_{\bar{\setD}\times \setE}(y,z) ,
\end{equation}
as required.
\end{proof}

%\begin{remark}\label{remjoa9lA}
%It follows from Theorem~\ref{fLew0a}, that Data relaxation, e.~g., of $\setD$ to $\bar{\setD}$, is preserved if the set $\setD$ is approximated by a sequence $(\setD_h)$ of material data sets Mosco-converging to $\setD$. Such approximations arise, %e.~g., if the material data set $\setD$ is sampled uniformly and increasingly finely, cf.~Theorem~\ref{wO2cri} and Theorem~\ref{0rIuth}.
%\end{remark}

\subsection{One-dimensional problems}
\label{seconed}
\begin{figure}[ht]
	\begin{subfigure}{0.49\textwidth}\caption{} \includegraphics[width=0.99\linewidth]{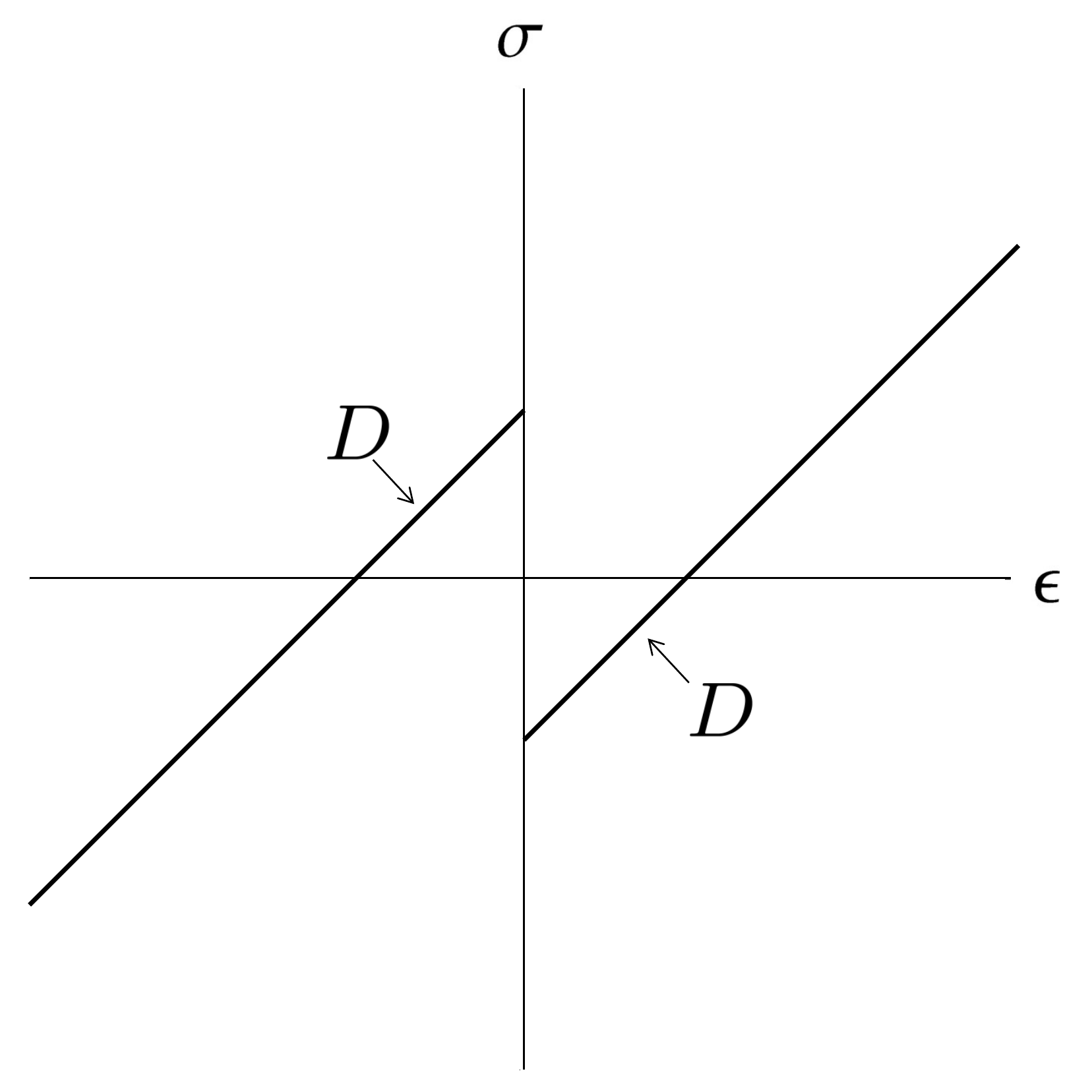}
	\end{subfigure}
	\begin{subfigure}{0.49\textwidth}\caption{} \includegraphics[width=0.99\linewidth]{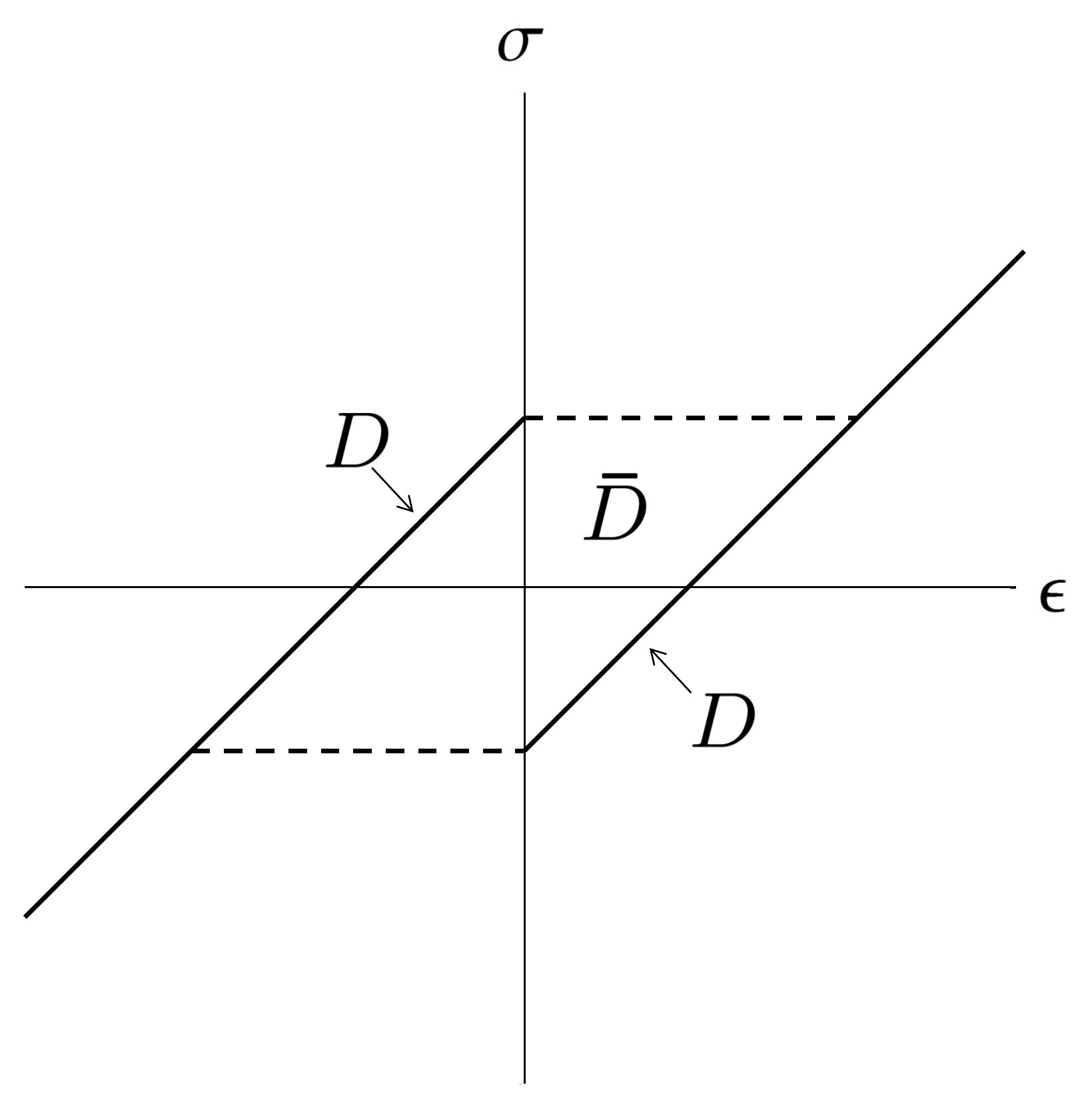}
	\end{subfigure}
	\caption{a) Two-phase material data set. b) Relaxed material data set.}
    \label{Voeb3e}
\end{figure}

We consider a one-dimensional elasticity problem defined over the domain $(0,1)$ with Dirichlet boundary conditions at both ends. We assume the body to be free of body forces. Other boundary conditions may be treated likewise. The addition of body forces represents a continuous perturbation of the problem and does not affect its relaxation.

In this case, the phase space is $Z = L^2(0,1)\times L^2(0,1)$ and, assuming $\Gamma_D=\{0,1\}$, the constraint set
$\setE$ defined in (\ref{s6oAzi}) reduces to
%with $L^2(\Omega;\mathbb{R}^{{n} \times {n}}_{\rm sym}) \times L^2(\Omega;\mathbb{R}^{{n}\times {n}}_{\rm sym})$, metrized by norm (\ref{nIus4o}). The constraint set ${\setE}$, Eq.~(\ref{s6oAzi}) will then consist
%of the elements of $Z$ that are compatible and in equilibrium.
\begin{equation}\label{F3oewO}
    {\setE}
    =
    \{
        \epsilon \in L^2(0,1),\ \int_0^1 \epsilon \, dx = \bar{\epsilon}
    \}
    \times
    \{
        \sigma = \bar{\sigma} \in \mathbb{R}
    \} .
\end{equation}
Here $\bar\epsilon\in\R$ is the macroscopic deformation, as given by the boundary conditions.
For definiteness, we specifically consider the data set
\begin{equation}\label{c8Leki}
    \setD = \{ y \in Z \, : \ y(x) \in \setD_{\rm loc} \text{ a.~e.~in } (0,1) \} ,
\end{equation}
with
\begin{equation}\label{flUP1a}
    \setD_{\rm loc}
    =
            \{(\epsilon, \mathbb{C} \epsilon + \sigma_0),
            \epsilon \leq 0 \}\cup
            \{(\epsilon, \mathbb{C} \epsilon - \sigma_0),
            \epsilon \geq 0 \},
\end{equation}
for some constants $\mathbb{C} > 0$, $\sigma_0 \geq 0$, cf.~Fig.~\ref{Voeb3e}a. We wish to elucidate the properties of the corresponding Data-Driven solutions. To this end, we introduce the relaxed material data set
\begin{equation}\label{Xo0giU}
    \bar{\setD} = \{ y \in Z \, : \ y(x) \in \bar{\setD}_{\rm loc} \text{ \ae~in } (0,1) \} ,
\end{equation}
with
\begin{align}\label{2rIefo}
    \bar{\setD}_{\rm loc}
    =&
            \{(\epsilon, \mathbb{C} \epsilon + \sigma_0),
            \epsilon \leq - 2 \sigma_0/\mathbb{C} \} \\
            &\cup
          \{  \{\epsilon\}\times[-\sigma_0,\mathbb{C}\epsilon + \sigma_0] ,
            - 2 \sigma_0/\mathbb{C} \leq \epsilon \leq 0\}  \notag\\
            &
            \cup
            \{\{\epsilon\}\times[\mathbb{C}\epsilon - \sigma_0, \sigma_0] ,
             0 \leq \epsilon \leq  2 \sigma_0/\mathbb{C}\}\notag\\
             &\cup
            \{(\epsilon, \mathbb{C} \epsilon - \sigma_0),
             \epsilon \geq 2 \sigma_0/\mathbb{C} \},\notag
\end{align}
cf.~Fig.~\ref{Voeb3e}b.

We proceed to show that $\bar{\setD}$ is indeed the relaxation of $\setD$.

\begin{thm}
Let $Z = L^2(0,1)\times L^2(0,1)$, ${\setE} \subset Z$ as in (\ref{F3oewO}), $\setD$ as in (\ref{c8Leki}) and (\ref{flUP1a}) and $\bar{\setD}$ as in (\ref{Xo0giU}) and (\ref{2rIefo}).
%    Let $(D_h) \subset Z$ be the constant sequence at $D$.
Then, $\bar{\setD} \times \setE= K(\topoD){-}\lim_{h\to\infty} \setD\times \setE$.
\end{thm}

\begin{proof}
i) We first verify the equi-transversality condition. Let $y=(\alpha,\beta)\in\setD$, $z=(\eps,\sigma)\in\setE$, and
\begin{equation}
    \|y-z\|\ge c \|\alpha-\eps\|_{L^2}+c\|\beta-\sigma\|_{L^2}.
\end{equation}
Since $\sigma$ is constant, $\|y-z\|$ controls the oscillation of $\beta$. From $y\in \setD$ we deduce  $|\alpha-\C\beta|\le \sigma_0$ pointwise, hence also the oscillation of $\alpha$ is controlled. The term $\|\alpha-\eps\|_{L^2}$ controls the distance of the average of $\alpha$ from $\bar\eps$. Therefore
\begin{equation}
    \|y-z\|\ge c(\|y\|+\|z\|) - b
\end{equation}
with $b$ depending on $\C$, $\sigma_0$, $\bar\eps$.

ii) Let $(y_h,z_h)$ be a sequence in $Z\times Z$ with limit $(y,z) = {\topoD}{-}\lim_{h\to\infty}(y_h,z_h)$. We need to show that
\begin{equation}
    \liminf_{h\to\infty}
    \Big( I_{\setD}(y_h) + I_{\setE}(z_h)  \Big)
    \geq
    I_{\bar{\setD}}(y) + I_{\setE}(z).
\end{equation}
By $\topoD$-convergence, we have: $y_h \rightharpoonup y$; $z_h \rightharpoonup z$; and $y_h - z_h \to y - z$.
It is enough to consider the case $y_h \in \setD$, $z_h \in {\setE}$, hence $z \in {\setE}$ by the weak closedness of ${\setE}$ in $Z$. Therefore, it remains only to verify that $y \in \bar{\setD}$.

Let $y_h = (\alpha_h,\beta_h)$, $y = (\alpha,\beta)$, $z_h = (\epsilon_h,\sigma_h)$ and $z = (\epsilon,\sigma)$. The convergences above then give $\alpha_h \rightharpoonup \alpha$ in $L^2(0,1)$; $\beta_h \rightharpoonup \beta$ in $L^2(0,1)$; $\epsilon_h \rightharpoonup \epsilon$ in $L^2(0,1)$; $\sigma_h \weakto \sigma$ in $L^2(0,1)$; $\alpha_h - \epsilon_h \to \alpha - \epsilon$  in $L^2(0,1)$; and $\beta_h - \sigma_h \to \beta - \sigma$ in $L^2(0,1)$.
Since $z_h\in \setE$ we have $\sigma_h=\bar\sigma_h\in\R$. Thus $\sigma=\bar\sigma\in\R$ and $\beta_h \to \beta$ in $L^2(0,1)$.

Since $(\alpha_h,\beta_h)\in \setD_\loc$ we have $\beta_h=\C \alpha_h +\sigma_0 \chi_h$, for some $\chi_h:(0,1)\to\{\pm 1\}$ that obeys $\chi_h=-1$ on the set $\{\beta_h>\sigma_0\}$, $\chi_h=1$ on the set $\{\beta_h<-\sigma_0\}$. It is immediate that $\beta=\C \alpha +\sigma_0 \chi$, where $\chi\in L^\infty((0,1);[-1,1])$ is the weak limit of $\chi_h$. Further, from $(\chi_h+1)(\beta_h-\sigma_0)\le 0$ pointwise and the strong convergence of $\beta_h$ we deduce $(\chi+1)(\beta-\sigma_0)\le 0$ almost everywhere, hence $\chi=-1$ almost everywhere on the set $\{\beta>\sigma_0\}$. Analogously one shows that $\chi=1$ \ae on  $\{\beta<-\sigma_0\}$. Hence, $(\alpha(x),\beta(x)) \in \bar{\setD}_{\rm loc}$ for a.~e.~$x \in (0,1)$ and $y = (\alpha,\beta) \in \bar{\setD}$.

iii) Let $(y,z) \in Z\times Z$. We need to show that there exists a sequence $(y_h,z_h)$ in $Z\times Z$ with $(y,z) = {\topoD}{-}\lim_{h\to\infty}(y_h,z_h)$ such that
\begin{equation}\label{diaM0E}
    \lim_{h\to\infty}
    \Big( I_{\setD}(y_h) + I_{\setE}(z_h)  \Big)
    \le
    I_{\bar{\setD}}(y) + I_{\setE}(z)  .
\end{equation}
We can suppose that $(y,z) \in \bar{\setD}\times {\setE}$.
Let $\eta>0$. Then for almost every $x\in (0,1)$ there is $\delta_x>0$ such that
\begin{equation}
y(x)\in \bar\setD_\loc \hskip5mm\text{ and } \hskip5mm \frac{1}{2\delta_x} \int_{x-\delta_x}^{x+\delta_x} |y-y(x)|^2 dx' <\eta.
\end{equation}
We cover almost all of $(0,1)$ by countably many such segments, $I_i=(x_i-\delta_i, x_i+\delta_i)$, and construct a function $y^\eta\in \bar\setD$ that is constant on each of these segments and obeys $\|y^\eta-y\|_{L^2}<\eta$.

We consider one of the segments and set $(\alpha_i,\beta_i)=y(x_i)\in\bar \setD_\loc$. By the definition of $\bar\setD_\loc$ there is $\lambda_i\in[0,1]$ such that
\begin{equation}
 \alpha_i = \lambda_i \C^{-1}(\beta_i -\sigma_0) + (1-\lambda_i) \C^{-1}(\beta_i +\sigma_0).
\end{equation}
(if $\beta_i>\sigma_0$ then necessarily $\lambda_i=0$, if $\beta_i<-\sigma_0$ then necessarily $\lambda_i=1$). Let $\alpha_h^{i,\eta}\in L^\infty(I_i,\{\C^{-1}(\beta_i +\sigma_0),\C^{-1}(\beta_i -\sigma_0)\})$ be a sequence that converges weakly to $\alpha_i$ and such that each $\alpha_h^{i,\eta}$ has average $\alpha_i$ (over the domain $I_i$). We set $\alpha_h^\eta=\alpha_h^{i,\eta}$ on $I_i$. One can then verify that $\alpha_h^\eta$ is bounded in $L^2$ uniformly in $h$ and $\eta$ and converges weakly as $h\to\infty$ to $\alpha^\eta$. Finally, we let $(\eps,\sigma)=z$, define
\begin{equation}
 \eps_h^\eta = \eps + \alpha_h^\eta -\alpha^\eta
\end{equation}
and set $z_h^\eta=(\eps_h^\eta, \sigma)$.
Then $z_h^\eta\in \setE$ and, denoting by  $\alpha$ the first component of $y$,
\begin{equation}
\|(\alpha_h^\eta-\eps_h^\eta)-(\alpha-\eps)\|_{L^2}=  \|\alpha^\eta-\alpha\|_{L^2} \le
\eta.
\end{equation}
Taking a diagonal subsequence we obtain  $(y_h,z_h) \in Z\times Z$ such that $y_h \in \setD$, $z_h \in {\setE}$ and $(y,z) = {\topoD}{-}\lim_{h\to\infty}(y_h,z_h)$, whereupon (\ref{diaM0E}) reduces to
\begin{equation}
    \lim_{h\to\infty}
     \|y_h-z_h\|
    =
    \|y-z\| ,
\end{equation}
which is indeed satisfied by the $\topoD$-convergence of $(y_h,z_h)$ to $(y,z)$.
\end{proof}

We note that the relaxed Data-Driven problem differs markedly from the classical relaxation of the two-well problem. Indeed, the classical variational formulation deals with minimizing $\int_0^1 W(u') dx$ over all $u\in H^1((0,1))$ with $u(1)-u(0)=\bar\epsilon$,
where the energy density takes the form
\begin{equation}
 W(\eps) = \min\{ \frac12\C (\eps+\C^{-1}\sigma_0)^2, \frac12\C (\eps-\C^{-1}\sigma_0)^2 \}.
\end{equation}
The relaxation of this scalar problem is obtained  replacing $W$ by its convex envelope,
\begin{equation}
 W^{**}(\eps) =
 \begin{cases}
  \frac12\C (\eps+\C^{-1}\sigma_0)^2 & \text{ if } \eps<-\sigma_0/\C\\
  0 & \text{ if } -\sigma_0/\C\le \eps\le\sigma_0/\C\\
  \frac12\C (\eps-\C^{-1}\sigma_0)^2 & \text{ if } \eps>\sigma_0/\C\,.
 \end{cases}
\end{equation}
This convex envelope corresponds to the data set
\begin{equation}\begin{split}
    \bar{\setD}'_{\rm loc}
    =&
            \{(\epsilon, \mathbb{C} \epsilon + \sigma_0),
            \epsilon \leq - \sigma_0/\mathbb{C}  \}   \\
      &      \cup \{(\epsilon,0),
             - \sigma_0/\mathbb{C} \leq \epsilon \leq \sigma_0/\mathbb{C}\},\\
& \cup            \{(\epsilon, \mathbb{C} \epsilon - \sigma_0),
             \epsilon \geq \sigma_0/\mathbb{C} \}.
                \end{split}
\end{equation}
which is markedly different from the Data relaxation $\bar\setD_\loc$.

Interestingly, the Data relaxed material data set (\ref{2rIefo}) has the 'flag' form that observed experimentally in materials undergoing displacive phase transitions when tested under cyclic loading, including unloading/reloading from partially transformed states, Fig.~\ref{KlUql9}.

\begin{figure}[ht]
	\begin{subfigure}{0.49\textwidth}\caption{} \includegraphics[width=0.99\linewidth]{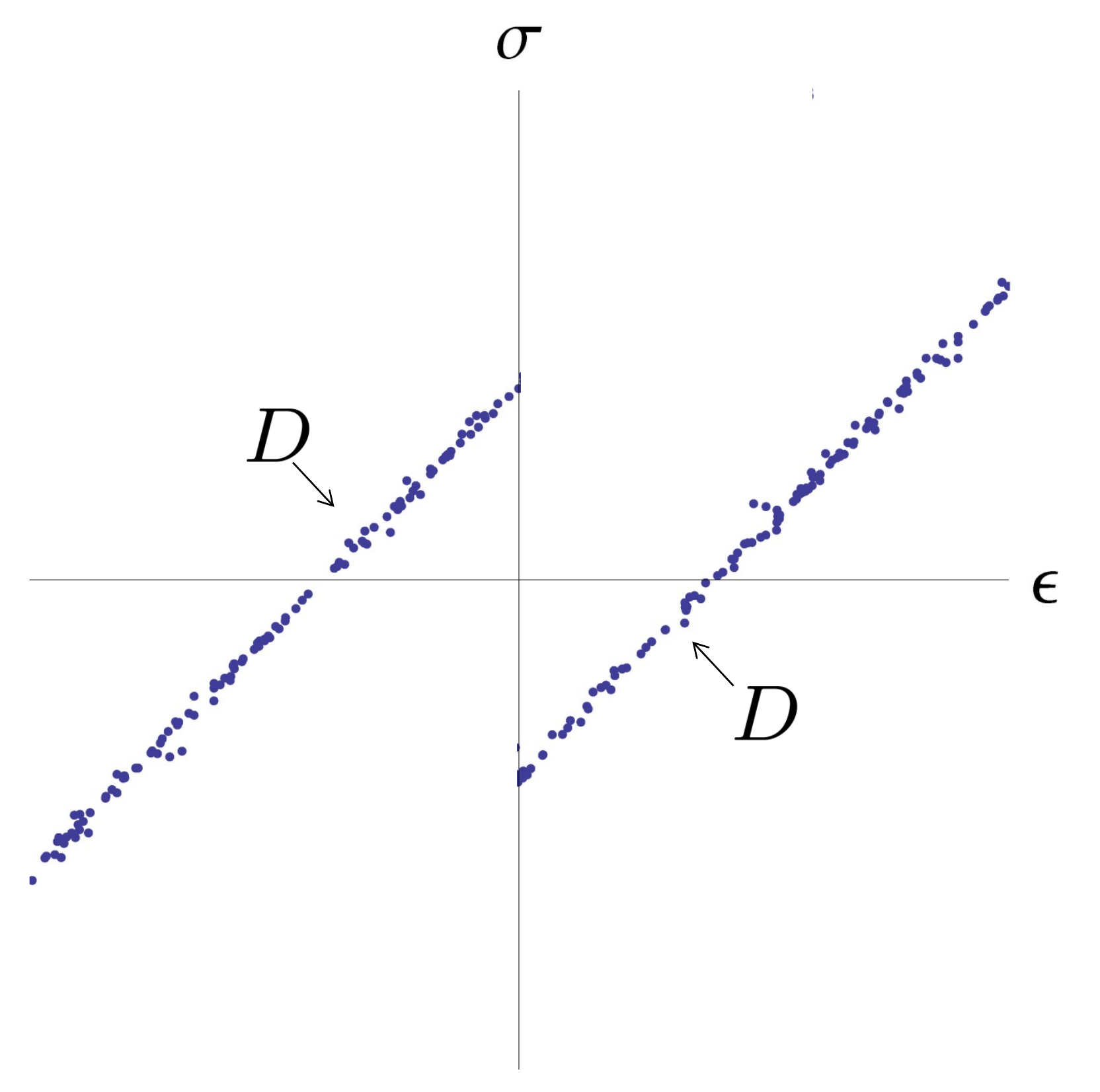}
	\end{subfigure}
	\begin{subfigure}{0.49\textwidth}\caption{} \includegraphics[width=0.99\linewidth]{Bar1.pdf}
	\end{subfigure}
	\caption{a) Sampled two-phase material data set. b) Relaxed material data set.}
    \label{jOa8hl}
\end{figure}

We may also characterize directly the distance-minimizing solutions. Indeed, the Data-Driven problem reduces to minimizing
\begin{equation}
    \Big\{
        d^2(z,\setD),
        \ z = (\epsilon,\bar{\sigma}) ,
        \ \epsilon \in L^2(0,1),
        \ \bar{\sigma} \in \mathbb{R} ,
        \ \int_0^1 \epsilon(x) \, dx = \bar{\epsilon}
    \Big\} .
\end{equation}
We write $\setD=\setD_+\cup \setD_-$, where $\setD_\pm=\{(\eps, \C\eps\mp\sigma_0): \pm\eps\ge 0\}$.
We subdivide $(0,1)$ into $\omega_+=\{x: d(z(x),\setD_+)<d(z(x),\setD_-)\}$ and $\omega_-=(0,1)\setminus\omega_+$.
By convexity and Jensen's inequality, $\epsilon$ may be taken to be a constant in each of them, and the problem reduces to minimizing
\begin{equation}
\begin{split}
    \Big\{
	  \lambda_A d^2((\eps_A,\bar\sigma),\setD_+) +\lambda_B d^2((\eps_B,\bar\sigma),\setD_-) :&
        \  \lambda_A+\lambda_B=1, \ \lambda_{A,B}\ge 0,\\
        &
        \bar{\sigma} \in \mathbb{R} , \
        \ \lambda_A\eps_A+\lambda_B\eps_B = \bar{\epsilon}
    \Big\} .
    \end{split}
\end{equation}
It follows that the minimum is zero  if and only if $(\bar{\epsilon},\bar{\sigma}) \in \bar{\setD}_{\rm loc}$, in agreement with Theorem~\ref{wO2cri}.

We note that, by Theorem~\ref{fLew0a}, cf.~also Remark~\ref{remarklocalapprox}, the same Data relaxation is obtained under
uniform sampling of $\setD_{\rm loc}$, as in Fig.~\ref{4iAcro}. Also, a similar analysis shows that $\bar{\setD}_{\rm loc} = \setD_{\rm loc}$ when $\setD_{\rm loc}$ is a monotonic graph. Therefore, monotonic graphs are stable with respect to Data relaxation, as expected.

\subsection{The multidimensional two-well problem}
\label{sectwowell}

We illustrate the set-valued character of Data relaxation in multiple dimensions with the aid of the  two-well problem with equal elastic moduli. This two-well problem of  linearized elasticity has been studied by many authors, including in particular Khachaturyan \cite{Khachaturyan1967some, KhachaturyanShatalov1969theory, Khachaturyan1983}, Roitburg \cite{Roitburd69, Roitburd78} and  Kohn \cite{RN26}, who obtained the classical relaxation of the problem. Again, we restrict attention to linearized kinematics and identify the phase space ${Z}$ with $L^2(\Omega;\mathbb{R}^{{n} \times {n}}_{\rm sym}) \times L^2(\Omega;\mathbb{R}^{{n}\times {n}}_{\rm sym})$ metrized by norm (\ref{nIus4o}). We recall that the constraint set ${\setE}$, eq.~(\ref{s6oAzi}), consists of the elements of $Z$ that are compatible and in equilibrium.

Let $a,b \in \mathbb{R}^{{n}\times {n}}_{\rm sym}$. Given an elasticity tensor
\begin{equation}\label{fOa8lA}
    \mathbb{C} \in L(\mathbb{R}^{{n}\times {n}}_{\rm sym}) ,
    \quad
    \mathbb{C}^T = \mathbb{C} ,
    \quad
    \mathbb{C} > 0 ,
\end{equation}
we consider a local material data set of the form
\begin{align}
    \setD_{\rm loc}
    =&
        \{(\epsilon,\mathbb{C} (\epsilon - a)): \eps\in\R^{n\times n}_\sym \text{ with } W(\epsilon - a) \le W(\epsilon - b)\} \\
        &\cup
        \{(\epsilon, \mathbb{C} (\epsilon - b)): \eps\in\R^{n\times n}_\sym \text{ with }
       W(\epsilon - b) \le W(\epsilon - a) \},
\notag
\end{align}
where we write
\begin{equation}\label{eqdefweps}
    W(\epsilon) = \frac{1}{2} \mathbb{C} \epsilon \cdot \epsilon .
\end{equation}
Equivalently,
\begin{equation}\label{5Leqie}
    \begin{split}
    \setD_{\rm loc}
    =& \{ (\epsilon, \mathbb{C} (\epsilon - a) ):\epsilon\in\R^{n\times n}_\sym\text{ with } \mathbb{C} \epsilon \cdot ( b - a ) \le W(b) - W(a) \}
    \\
    &\cup  \{ (\epsilon,    \mathbb{C} (\epsilon - b))  :\epsilon\in\R^{n\times n}_\sym\text{ with }
 \mathbb{C} \epsilon \cdot ( b - a ) \ge  W(b) - W(a)\}.
 \end{split}
\end{equation}
This local material data set represents a material that can be in one of two phases characterized by transformation strains $a$ and $b$. The classical variational formulation of the problem deals then with the minimization of $\int_\Omega \min \{W(e(u)-a), W(e(u)-b)\} dx$.

After a translation, we may and will assume without loss of generality that
$a = -b$.
Then
\begin{equation}  \label{eq:def_Dloc1}
\setD_\loc = \setD^+_\loc \cup \setD^-_\loc
\end{equation}
where
\begin{align}   \label{eq:def_Dloc2}
  \setD^+_\loc :=&\,  \{ (\C^{-1} \sigma +  b, \sigma) : \sigma \in \R^{n \times n}_\sym, \,  \sigma \cdot b \ge - \C b \cdot b\}, \\
    \label{eq:def_Dloc3}
\setD^-_\loc :=&\,  \{ (\C^{-1} \sigma - b, \sigma) : \sigma \in \R^{n \times n}_\sym, \,  \sigma \cdot b \le   \C b \cdot b   \} = -\setD^+_\loc.
\end{align}
For $c, \nu \in  \R^n$, we define the symmetrized tensor product $c \odot \nu$ by $(c\odot\nu)_{ij} = (c_i\nu_j+c_j\nu_i)/2$. 

Our main result is the following.
\begin{thm}
Consider the global material data sets
\begin{equation}\label{eqdefsetDtwowell}
    \setD = \{ y \in Z \, : \ y(x) \in \setD_{\rm loc} \text{ \ae~in } \Omega \} ,
\end{equation}
where $\setD_{\rm loc}$ is given by    \eqref{eq:def_Dloc1}--\eqref{eq:def_Dloc3},
%(\ref{5Leqie}), 
and
\begin{equation}\label{eqdefsetDbartwowell}
    \bar{\setD} = \{ y \in Z \, : \ y(x) \in \bar{\setD}_{\rm loc} \text{ \ae~in } \Omega \} ,
\end{equation}
with $\bar{\setD}_{\rm loc}$ given by

\begin{equation}\label{eqdefDlocbar3_bis}
    \begin{split}
  \bar  \setD_{\rm loc}
    =&  \setD_\loc  %\\
    %&
    \cup  \{ (\C^{-1}\sigma + \mu b , \sigma)   \text{ for }\mu \in(-1,1),  \sigma\in\R^{n\times n}_\sym,\\
&\hskip1.5cm   | \sigma\cdot b  +  \alpha_- \, \mu|  \le \C b \cdot b  -  \alpha_-  \}.
    \end{split}
\end{equation}
The parameter $\alpha_-$ is defined by 
\begin{equation}  \label{eq:def_alpha_minus}
 \alpha_- :=  \min \left\{\C \big(c \odot \nu -   b \big) \cdot \big( c \odot \nu -   b \big) :   c \in \R^n, \nu \in S^{n-1} \right\}
\end{equation}
%
%
%% Old version with general $a$ and $b$
%\begin{equation}\label{eqdefDlocbar3}
%    \begin{split}
%  \bar  \setD_{\rm loc}
%    =&  \setD_\loc  %\\
%    %&
%    \cup  \{  \C^{-1}\sigma + \lambda a+(1-\lambda) b,\,  \sigma)   \text{ for }\lambda\in(0,1),  \sigma\in\R^{n\times n}_\sym,\\
%&\hskip1.5cm   | \sigma\cdot(b-a) +  (1 - 2 \lambda) 2 \alpha_- |  \le W(a-b) - 2 \alpha_-  \}.
%    \end{split}
%\end{equation}
%where
%\begin{equation}
% \alpha_- = \frac14 \min \left\{\C \big(2c \odot \nu -   (b-a) \big) \cdot \big( 2c \odot \nu -   (b-a) \big) :   c \in \R^n, \nu \in S^{n-1} \right\}
% \end{equation}
and $\Omega$, $Z$, $\setE$ are  as specified above. Then, $\bar{\setD} \times \setE$ is the Data relaxation of $\setD\times \setE$, in the sense that $\bar{\setD} \times \setE= K(\topoD){-}\lim_{h\to\infty} \setD\times \setE$.
\end{thm}
\begin{proof}
The assertion follows from Theorem \ref{theoub} and Theorem \ref{theotwowellllwerb} below.
\end{proof}

It is no coincidence that the relaxed set $\overline \setD_\loc$ can be described in terms of the two parameters
$\sigma \cdot b$ and $\mu = (\C^{-1} \sigma - \eps) \cdot b/|b|^2$. Indeed on the one hand $\setD_\loc$ (and hence $\overline \setD_\loc$) is contained
in the linear subspace 
\begin{equation} \label{eq:span_Dloc}
L := \Span \setD_\loc = \{ (\C^{-1} \sigma + \mu b , \sigma) : \sigma \in \R^{n \times n}_\sym, \, \mu \in \R \}.
\end{equation}
On the other hand $\setD_\loc$ (and hence $\overline \setD_\loc$) is invariant under translations by elements of the linear subspace
\begin{equation} \label{eq:inner_space_Dloc}
L' := \{  (\C^{-1} \sigma, \sigma) : \sigma \cdot b = 0\} \subset L. 
\end{equation}
The quotient $L/ L'$ is two-dimensional and described by the parameters $\sigma \cdot b$ and $\mu$. 
A sketch of the set $\overline \setD_\loc$ in  the $(\sigma \cdot b, \mu)$ plane is given in Fig.~\ref{fi:separation}.

\medskip

Also note that
\begin{equation} \label{eq:bound_alpha_minus}
 \alpha_- < \C b \cdot b
\end{equation}
and thus $\overline \setD_\loc$ always has non-empty interior in $L$. 
Indeed clearly $\alpha_- \le \C b \cdot b$. If equality holds then the function
$f(s) = \C (b - s c \odot \nu) \cdot (b - s c \odot \nu)$ has a minimum at $s=0$
and thus $\C b \cdot (c \odot \nu) = 0$ for all $c \in \R^n$ and $\nu \in S^{n-1}$.
Therefore $\C b= 0$ and hence $b=0$, a contradiction.

\begin{remark}[Energy wells of unequal height] \label{re:unequal_height}
Energetically, the set $D_{\rm loc}$ above corresponds to two-wells of equal height, \ie, to the energy $\tilde W(\epsilon) = \min\{ W(\epsilon - a), W(\epsilon-b)\}$. One can also consider two-wells of unequal height, \ie, $\hat W(\epsilon) = \min\{ W(\epsilon - a), W(\epsilon-b) + w\}$ for some $w \in \R$. This corresponds to the set
\begin{align}
   \hat{\setD}_{\rm loc}
   =&
       \{(\epsilon,\mathbb{C} (\epsilon - a): \eps\in\R^{n\times n}_\sym \text{ with } W(\epsilon - a) \le W(\epsilon - b) + w\} \\
       &\cup
       \{(\epsilon, \mathbb{C} (\epsilon - b): \eps\in\R^{n\times n}_\sym \text{ with }
      W(\epsilon - b) +w \le W(\epsilon - a) \}.
\notag
\end{align}
This situation can be reduced to the case of wells of equal height by a translation in $(\epsilon, \sigma)$ space. Indeed, if we set
\begin{equation}
 \lambda = \frac{w}{(b-a) \cdot \mathbb{C}(b-a)}
\end{equation}
then
\begin{equation}
 W(\epsilon + \lambda (b-a) -a) < W(\epsilon + \lambda(b-a) -b) + w \, \,  \Longleftrightarrow \, \,
W(\epsilon -a) < W(\epsilon -b)
\end{equation}
and thus
\begin{equation}
 \hat{\setD}_{\rm loc} = \setD_{\rm loc} + (\lambda (b-a), \lambda \mathbb{C}(b-a)).
\end{equation}
Therefore, there is no loss of generality in considering wells of equal height.
\end{remark}

 \bigskip

\begin{example}[Compatible wells]
We have $\alpha_- = 0$ if an only if the wells are compatible, \ie, if there exist $c \in \R^n$ and $\nu \in S^{n-1}$ such that $b = c \odot \nu$. 
In this case $\overline \setD_\loc$ has the same flag shaped form as in the one dimensional case,  cf.~Fig.~\ref{Voeb3e}b.
More precisely, for $\bar \sigma \in \R^{n \times n}_\sym$ with $|\bar \sigma \cdot b| \le \C b \cdot b$ the intersection
$\overline \setD_\loc  \cap (\R^{n \times n}_\sym \times \{ \bar \sigma\})$ consists of the segment $[\C^{-1} \bar \sigma - b, \C^{-1} \bar \sigma + b] \times \{\bar \sigma\}$
while for $\bar \eps \in   \R^{n \times n}_\sym$ with $\bar \mu := \C \eps \cdot b / \C b \cdot b \in (-2,2)$ the set
$ \overline \setD_\loc  \cap ( \{\bar \eps\} \times \R^{n \times n}_\sym)$ consists of the segment
$\{\bar \eps \} \times [ \C \bar \eps + \mu_- \C b, \C \bar \eps + \mu_+ \C b]$ where
$[\mu_-, \mu_+] = [\bar \mu -1, 1]$ if $\bar \mu \in [0, 2)$ and
$[\mu_-, \mu_+] = [-1, \bar \mu +1]$ if $\bar \mu \in (-2,0]$. Here we denoted the segment between $x$ and $y$ by $[x,y] := \{ \lambda x + (1-\lambda) y : \lambda \in [0,1]\}$.
%% the following is not quite correct an was hence removed, SM Aug 16
% the axis now correspond to $\sigma \cdot b$ and $\C^{-1} \eps \cdot b$).
\end{example}

% The incompatible example is now explained not only for $b = \Id$ but for a general 2d matrix with $\det b > 0$. 
% SM Aug 16.
\begin{example}[Incompatible wells] Consider the case $n=2$, $\C = \Id$. In a suitable orthonormal basis we can assume that
$b = \begin{pmatrix} b_1 & 0 \\ 0 & b_2 \end{pmatrix}$.
If $b_1 b_2 = \det b \le 0$ then $b$ is compatible and $\alpha_-=0$. 
If $b_1 b_2 > 0$ then $\C b \cdot b = b_1^2 + b_2^2$ and $\alpha_- = \min(b_1^2, b_2^2)$ (this follows from 
\eqref{eq:extremal_alpha} below: for $n=2$ we have $\bar \sigma = 0$ or $c$ parallel to $\nu$). 
Assume for definiteness that $0 <  b_1 \le b_2$. If $\eps = \sigma + \mu b$ then
$(\eps - \sigma) \cdot b = \mu |b|^2$. Thus  the condition $|\sigma \cdot b + \alpha_- \mu| \le |b|^2 - \alpha_-$
becomes after multiplication by $|b|^2$ and rearrangement
% dropped to keep things short.
%$$ | |b|^2  \sigma \cdot b  + b_1^2 (\eps \cdot b - \sigma \cdot b)| \le |b|^2 b_2^2$$
%or
$$ | b_2^2\,  \sigma \cdot  b+  b_1^2  \, \eps \cdot b| \le |b|^2\,   b_2^2$$
while the condition $|\mu| \le 1$ translates into
$$ | \sigma \cdot b - \eps \cdot b| \le  |b|^2.$$
For $b_1 = b_2$ the set $\overline \setD_\loc  \setminus \setD_\loc$ becomes a rotated square in the 
$(\tr \eps, \tr \sigma)$ plane, while in the vanishing incompatibility limit  $b_1 \downarrow 0$ the set $\overline \setD_\loc$ approaches the
flag shaped form, % in Fig.~\ref{Voeb3e}b, 
cf.~Fig.~ \ref{fi:incompatible_2well}.
\end{example}

\begin{figure}[ht] 
\begin{tikzpicture}
\draw[->][style=thin] (-5,0) -- (5,0) node[anchor=north]    {$\eps \cdot b$}; %x-axis
\draw[->][style=thin] (0,-4) -- (0,4) node[anchor=west]    {$\sigma \cdot b$}; % y-axis
\draw[style=thick]  (-5,-3) -- (0,2); 
\draw[style=thick]  (0,-2) -- (5,3);
\draw[style=dashed]  (0,2) -- (3,1);
\draw[style=dashed] (-3,-1) -- (0,-2)  ;
\draw (2,.1) -- (2,-.1)    node[anchor=north] {$\quad  \, \, b_1^2 + b_2^2$};
\draw (-.1, 2) -- (.1, 2)  node[anchor=west] {$b_1^2 + b_2^2$} ;
\end{tikzpicture} 
\caption{The set $\overline \setD_{\loc}$ for the incompatible case  $n=2$, 
$-a=b =\mathop{\mathrm{diag}}(b_1, b_2)$ 
%$\begin{pmatrix} b_1  & 0 \\ 0 & b_2 \end{pmatrix}$
% \begin{pmatrix} ... does not work in the caption environment 
% there is a workaround using \newsavebox{bla}
% \savebox{bla}{  \begin{pmatrix} .... \end{pmatrix} }
% in the preamble and then
% \usebox{bla} in the caption
% but I did not like that. 
 with $0 <  b_1 \le b_2$ and $\C = \Id$.
The set $\setD_\loc$ corresponds to the two solid half lines. The set $\overline \setD_\loc$
is obtained by adding the  closed region  bounded by the solid and dashed lines.
The dashed lines have slope $-b_1^2/ b_2^2$. In the limit  $b_1 \downarrow 0$  one obtains
again the flag-shaped region which arises in the compatible case $b_1 = 0$. If $b_1 = b_2$ the region 
bounded by the solid and shaded lines becomes a square, rotated by 45 degrees with respect to the axis.}
 \label{fi:incompatible_2well}
\end{figure}
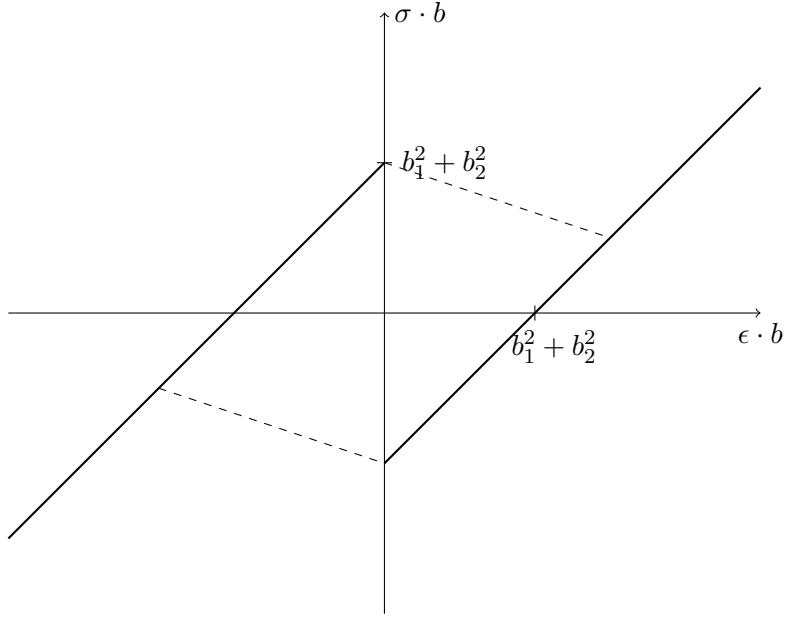

%% Old version of the incompatible example   % SM Aug 16
%\begin{example}
%Consider the case $n=2$, $\C = \Id$. In a suitable orthonormal basis we can assume that $b = \begin{pmatrix} b_1 & 0 \\ 0 & b_2 \end{pmatrix}$. If $b_1 b_2 = \det b \le 0$ then $b$ is compatible and $\alpha_-=0$. If $b_1 b_2 > 0$ then $\C b \cdot b = b_1^2 + b_2^2$ and $\alpha_- = \min(b_1^2, b_2^2)$ (this follows from \eqref{eq:extremal_alpha} below: for $n=2$ we have $\bar \sigma = 0$ or $c$ parallel to $\nu$). Thus, the strongest restriction on $\overline \setD_\loc$ arises if $b$ is proportional to the identity. If $b = \Id$ then $\alpha_- =1$ and elementary manipulations show that the set $\overline \setD_\loc \setminus \setD_\loc$ is given by the constraints $\eps = \sigma + \mu \Id$, $ |\tr(\eps + \sigma) |\le 2$ and $|\tr(\eps - \sigma)| < 2$. Thus in the $(\tr \eps, \tr \sigma)$ plane  instead of a flag we get a square, rotated by 45 degrees with respect to the coordinate axis.
%\end{example}

\bigskip

We begin by considering microstructures in the form of simple laminates consisting of two phases, labeled $+$ and $-$. We say that $(\eps_-, \sigma_-)$ and $(\eps_+, \sigma_+)$ are compatible across a normal $\nu \in S^{n-1}$
$$ \eps_+ - \eps_- = c\odot  \nu, \quad (\sigma_+ - \sigma_-) \nu = 0. $$
These conditions enforce compatibility and equilibrium of a piecewise constant map across an  interface with normal $\nu$. 
To express these conditions concisely, we introduce the cone
\begin{equation} \label{eq:lambda_cone}
\Lambda := \{ (\eps, \sigma) \in \R^{n \times n}_\sym \times \R^{n \times n}_\sym : \eps = c \odot \nu, \, \sigma \nu = 0, \, c \in \R^n, \, \nu \in S^{n-1} \}.
\end{equation}
We look for pairs $z_-, z_+ \in \setD_\loc$ with $z_+ - z_- \in \Lambda$.
%
% It is convenient to consider the linear hull
Recall that 
\begin{equation*} 
L := \Span \setD_\loc = \{ (\C^{-1} \sigma + \mu b , \sigma) : \sigma \in \R^{n \times n}_\sym, \, \mu \in \R \}
\end{equation*}
It will be convenient to characterize  $L \cap \Lambda$ first. 
Geometrically, the heart of the matter is that 
the canonical projection $\pi: L \to L/L'$ maps $\Lambda \cap L$ to a convex cone
\begin{equation}  \label{eq:geometric_heart}
 \pi(\Lambda \cap L) = \bigcup_{\alpha \in [\alpha_-, \alpha_+]} \{ (\sigma \cdot b, \mu)  : \sigma \cdot b + \alpha \mu = 0\},
 \end{equation}
see \eqref{eq:separate_Lambda_1} and \eqref{eq:proj_Lambda_cap_L} below.
%Note that $L$ is isomorphic to $\R^{n \times n}_\sym \times \R$ and 
%we can use the pair $(\sigma, \mu)$ as coordinates in $L$. 
%In these coordinates, $\pm \setD_\loc^\pm = H \times \{1\}$ where $H = \{ \sigma : \sigma \cdot b \ge -\C b \cdot b\}$.

\begin{lemma} \label{le:Lambda_cap_L}
For each $\nu \in S^{n-1}$, there exists one and only one $\hat c(\nu) \in \R^n$ such that
\begin{equation} \label{eq:eqn_hatc}
\C\big(  \hat c(\nu) \odot \nu - b\big) \nu = 0.
\end{equation}
Moreover, for all $\nu \in S^{n-1}$,
\begin{equation} \label{eq:min_hatc}
\C\big(  \hat c(\nu) \odot \nu - b\big) \cdot \big(  \hat c(\nu) \odot \nu - b\big)
= \min_{c \in \R^n}
\C\big(  c \odot \nu - b\big) \cdot \big(  c \odot \nu - b\big)
\end{equation}
and
\begin{equation} \label{eq:hat_sigma_cdot_b}
\C\big(  \hat c(\nu) \odot \nu - b\big) \cdot b = - \C\big(  \hat c(\nu) \odot \nu - b\big) \cdot \big(  \hat c(\nu) \odot \nu - b\big).
\end{equation}
In particular,
\begin{equation} \label{eq:Lambda_cap_L}
\Lambda \cap L = \R \left\{   \Big( \hat c(\nu) \odot \nu, \, \C\big( \hat c(\nu) \odot \nu - b\big) \Big)  :  \nu \in S^{n-1} \right\}
\end{equation}
The quantity $\alpha_-$, defined in \eqref{eq:def_alpha_minus}, satisfies
\begin{align}  \label{eq:alpha_minus} 
\alpha_- =& \,  \min_{\nu \in S^{n-1}} \C\big(  \hat c(\nu) \odot \nu - b\big) \cdot \big(  \hat c(\nu) \odot \nu - b\big),
\end{align}
Define
\begin{align}
\alpha_+ =& \,  \max_{\nu \in S^{n-1}}  \C\big(  \hat c(\nu) \odot \nu - b\big) \cdot \big(  \hat c(\nu) \odot \nu - b\big).
\label{eq:alpha_+}
\end{align}
Then,
\begin{equation}  \label{eq:separate_Lambda_1}
(\eps, \sigma)  \in \Lambda \cap L  \quad  \Longrightarrow \quad   -(\sigma \cdot b + \alpha_- \mu) (\sigma \cdot b + \alpha_+ \mu)  \ge 0 ,
\end{equation}
where $\mu$ is uniquely defined by the relation $\eps = \C^{-1} \sigma + \mu b.$ Conversely, for each $\alpha \in [\alpha_-, \alpha_+]$  and each $\mu \in \R$ there exists $\sigma \in \R^{n \times n}_\sym$
\begin{equation} \label{eq:proj_Lambda_cap_L}
 (\C^{-1} \sigma +  \mu b, \sigma) \in \Lambda \cap L  \quad \hbox{and} \quad    \sigma \cdot b +   \alpha \mu  = 0.
\end{equation}
Moreover, if $\C (b - \bar c \odot \bar \nu) \cdot (b - \bar c \odot \bar \nu) = \alpha_-$ for $\bar c \in \R^n$ and $\bar \nu \in S^{n-1}$ then
\begin{equation} \label{eq:extremal_alpha}
\bar \sigma \bar \nu = \bar \sigma \bar c = 0 \quad \hbox{where $\bar \sigma = \C (b - \bar c \odot \bar \nu)$.}
\end{equation}
\end{lemma}

\begin{proof}
For fixed $\nu \in S^{n-1}$, consider the map $g: \R^n \to \R$ given by
$$
    g(c) = \C\big(  c \odot \nu - b\big) \cdot \big(  c \odot \nu - b\big).
$$
Then, $g$ is strictly convex because $\C > 0$. Hence, $g$ has a unique minimizer
$\hat c(\nu)$ and variations of the form $c(s) = \hat c(\nu) + s \dot c$ show that
$\hat c(\nu)$ is uniquely characterized by the condition
$$
    \forall \dot c \in \R^n \quad \C\big(  \hat c(\nu)  \odot \nu - b\big) \cdot \big(  \dot c \odot  \nu \big).
$$
Since $\C\big(  \hat c(\nu)  \odot \nu - b\big)$ is symmetric, this is equivalent to
$$
    \C\big(  \hat c(\nu)  \odot \nu - b\big) \nu = 0.
$$
This proves the existence and uniqueness of a solution of \eqref{eq:eqn_hatc} as well as  the relation \eqref{eq:min_hatc}. The identity \eqref{eq:hat_sigma_cdot_b} follows since \eqref{eq:eqn_hatc} implies that
$$
    \C\big(  \hat c(\nu)  \odot \nu - b\big) \cdot \hat c(\nu)  \odot \nu  = 0.
$$

We next to prove \eqref{eq:Lambda_cap_L}. The inclusion $\supset$ is easy. Indeed, it follows directly from  $ \eqref{eq:eqn_hatc}$ that $\mu  \Big( \hat c(\nu) \odot \nu, \, \C \big( \hat c(\nu) \odot \nu - b\big)\Big) \in \Lambda$ for any $\mu \in \R$. If we set $\sigma =  \C\big( \hat c(\nu) \odot \nu - b\big)$, then also
$$
\mu \, \big (  \hat c(\nu) \odot \nu, \,  \sigma\big)
= \mu (\C^{-1} \sigma + b, \, \sigma) \in L.
$$

We now show the inclusion $\subset$. Clearly, $0$ belongs to the right hand side of \eqref{eq:Lambda_cap_L}. Thus, let $(\eps, \sigma) \in \Lambda \cap L \setminus \{0\}$. Then, $\eps = c \odot \nu$ and $\sigma \nu= 0$ for some $\nu \in S^{n-1}$ and some $c \in \R^n$ and, in addition, $\sigma = \C( c \odot \nu - \mu b)$. If $\mu = 0$, then the condition $\sigma \nu = 0$ implies that
$$
    0= \sigma \cdot (c \odot \nu) = \C (c \odot \nu ) \cdot (c \odot \nu)
$$
and, thus, $\eps = c \odot \nu = 0$ and $\sigma = 0$, a contradiction. Thus, we may assume that $\mu \ne 0$. Then division by $\mu$ gives
$$
    \C \big (\mu^{-1} c \odot \nu - b\big) \nu = 0.
$$
Since \eqref{eq:eqn_hatc} has a unique solution, we get $c = \mu \hat c(\nu)$ and, thus,
$$
    (\eps, \sigma)
    =
    \mu\,  \Big(\hat c(\mu) \odot \nu,\,   \C \big(\hat c(\nu) \odot \nu - b \big) \Big).
$$
This finishes the proof of  \eqref{eq:Lambda_cap_L}.

The identity \eqref{eq:alpha_minus} follows directly from 
\eqref{eq:min_hatc} and \eqref{eq:def_alpha_minus}.

To prove \eqref{eq:separate_Lambda_1}, let $(\eps, \sigma) \in \Lambda \cap L$. Then, $\sigma= \C (\eps - \mu b)$. If $\mu = 0$, then the identities $\eps = c \odot \nu$ and $\sigma \nu = 0$ yield $ \C (c \odot \nu) \cdot c \odot \nu =0$ and hence $\eps = \sigma = 0$, so that the desired relation holds. If $\mu \ne 0$, we may assume $\mu = 1$ since the assertion is invariant under the scaling  $(\sigma, \mu) \to (\sigma/ \mu', \mu/ \mu')$. By \eqref{eq:Lambda_cap_L}, there exists $\nu \in S^{n-1}$ such that $ \sigma  = \C( \hat c(\nu)  \odot \nu - b)$. Now \eqref{eq:hat_sigma_cdot_b}, \eqref{eq:min_hatc} and the definition of $\alpha_\pm$ give $ \sigma \cdot b  \in [- \alpha_+, -\alpha_-]$, which implies the desired assertion.

To prove \eqref{eq:proj_Lambda_cap_L} let
$$
    \hat \alpha(\nu)
    :=
    \min_{c \in \R^n}  \C\big(  c \odot \nu - b\big) \cdot \big(  c \odot \nu - b\big)
=  \C\big(  \hat c(\nu) \odot \nu - b\big) \cdot \big(  \hat c(\nu) \odot \nu - b\big).
$$
From the first identity one easily deduces that $\hat \alpha$ is continuous. Thus, by the intermediate value theorem for each $\alpha \in [\alpha_-, \alpha_+]$ there exists $\nu \in S^{n-1}$ such that $\hat \alpha(\nu) = \alpha$. Let $\sigma =  \mu \C\big(  \hat c(\nu)  \odot \nu - b\big)$ and $\eps = \mu \hat c(\nu)  \odot \nu$. Then, by \eqref{eq:Lambda_cap_L} we have $(\eps, \sigma) \in \Lambda \cap L$ and using \eqref{eq:hat_sigma_cdot_b} we get
$$
    - \sigma \cdot b = \mu \,  \hat \alpha(\nu) =\mu  \alpha.
$$

Finally, we prove \eqref{eq:extremal_alpha}. The condition $\bar \sigma \bar \nu = 0$  follows by minimizing in $c$ for fixed $\bar \nu$ and \eqref{eq:eqn_hatc} in combination with \eqref{eq:min_hatc}. To get the other condition, let $\dot \nu \in S^{n-1}$ with  $\dot \nu \cdot \bar  \nu = 0$ and consider variations $\nu(\theta) = \cos \theta\,  \nu + \sin \theta \, \dot \nu$. This gives $\bar \sigma \cdot (\bar c \odot \dot \nu )= 0$ and thus $\bar \sigma \bar c = \beta  \bar \nu$. The assertion follows since $\bar \nu \cdot (\bar \sigma \bar c) = \bar c \cdot \bar \sigma \bar \nu  =0$. This finishes the proof of the lemma.
\end{proof}

\begin{lemma}  \label{le:rank_one_connection} Let $z = (\eps, \sigma) \in \overline \setD_\loc \setminus \setD_\loc$. Then, there exist $z_- = (\eps_-, \sigma_-) \in \setD^-_\loc$, $z_+ = (\eps_+, \sigma_+)\in  \setD^+_\loc$ and $\lambda \in (0,1)$ such that
\begin{equation} \label{eq:laminate_compatible}
z_+ - z_- \in \Lambda  \quad  \hbox{and} \quad z = \lambda z_- + (1-\lambda) z_+.
\end{equation}
Moreover,\begin{equation}
|z_\pm| \le C (|z| + 1).  \label{eq:laminate_bound}
\end{equation}
\end{lemma}

\begin{proof}
It follows from the definition of $\overline \setD_\loc$ that
\begin{equation}  \label{eq:condition_eps_laminate}
 \eps =\C^{-1} \sigma + \mu b \quad \hbox{with} \quad  \mu \in (-1,1)
\end{equation}
and
\begin{equation} \label{eq:condition_sigma_laminate}
|  \sigma \cdot b + \mu \alpha_- | \le \C b \cdot b - \alpha_-.
\end{equation}
By Lemma \ref{le:Lambda_cap_L}, there exists $\nu \in S^{n-1}$ and $c = \hat c(\nu)$ such that
\begin{equation}   \label{eq:optimal_Lambda_conn}
\hat z =  (\hat \eps, \hat \sigma) := \Big (c  \odot  \nu, \, \C( c \odot  \nu - b)\Big) \in \Lambda
\quad \hbox{and} \quad
\hat \sigma \cdot b = -\alpha_-.
\end{equation}
Set
$$
    z_+ = z + (1- \mu) \hat z, \quad z_- = z  + (-1 - \mu) \hat z, \quad \lambda = \frac{1+\mu}{2}.
$$
Then, $\lambda \in(0,1)$, $z = \lambda z_- + (1-\lambda) z_+$ and $z_+ -z_- = 2 \hat z \in \Lambda$.

It only remains to show that $z_+ \in \setD^+_\loc$ and $z_- \in \setD^-_\loc$. We have $\hat \eps - \C^{-1} \hat \sigma = b$ and together with \eqref{eq:condition_eps_laminate} this gives  $\eps_+ - \C^{-1} \sigma_+ = b$. Using first  the last identity in \eqref{eq:optimal_Lambda_conn} and then  \eqref{eq:condition_sigma_laminate}, we get
$$
    \sigma_+ \cdot b  = \sigma \cdot b  - (1-\mu) \alpha_-   = \sigma \cdot b + \mu  \alpha _-  - \alpha_- \ge -\C b \cdot b
$$
and hence $z_+ \in \setD^+_\loc$. Similarly, one shows that $z_- \in \setD^-_\loc$. To show  \eqref{eq:laminate_bound} it suffices to note that $|z_\pm - z| \le 2 |\hat z|  \le C$ where $C$ depends only on $b$ and $\C$.
\end{proof}

We proceed to show that $\bar{\setD}$ is indeed the relaxation of $\setD$. We first show that every element in $\overline \setD \times \setE$ can be approximated in the sense of data convergence by elements of $\setD \times \setE$. Lemma \ref{le:rank_one_connection} is the key ingredient for the case of constant limit  maps. For the general case, we will then use a suitable covering and gluing argument.

\def\Rvierst{\R^{n^4}_*}
\def\apot{\varphi}
\def\bpot{\psi}

Such gluing constructions for vector fields that obey differential constraints are much easier if one works with the corresponding potentials. For the strain one uses the displacement field as potential. Analogously, for the stress one uses a stress potential $\apot$, which is related to the field $\sigma$ by $\sigma=\Div\Div \apot$.

Let $\Rvierst$ be the set of $\xi\in \R^{n\times n\times n\times n}$ such that
\begin{equation}\label{eqdefRvierst}
 \xi_{ijhk}=\xi_{jikh}=-\xi_{ihjk} \hskip5mm \text{ for all } i,j,k,h\in\{1, 2, \dots, n\}.
\end{equation}
For $\apot\in L^1_\loc(\R^n;\Rvierst)$, we define the distribution
\begin{equation}
 (\Div\Div \apot)_{ij} =  \partial_{h}\partial_{k} \apot_{ijhk},
\end{equation}
where a sum over the repeated indices $h$ and $k$ is implied. We observe that by (\ref{eqdefRvierst}) it follows that $\Div (\Div \Div \apot)=0$ and $\Div \Div \apot=(\Div \Div \apot )^T$. For $M\in\R^{n\times n}_\sym$ we define $\apot^M:\R^n\to\Rvierst$ by
\begin{equation}\label{eqdefaM}
\begin{split}
 \apot^M(x)_{ijhk}=\frac1{n(n-1)} &\bigl( M_{ij}x_hx_k+M_{hk}x_ix_j
 - M_{ih}x_jx_k - M_{kj}x_hx_i \bigr).
\end{split}
\end{equation}
A straightforward computation shows that $\Div\Div \apot^M=M$, with $|\apot^M|(x)\le  2|x|^2|M|$, $|D\apot^M|(x)\le 4|x|\, |M|$, $|D^2\apot^M|(x)\le 4|M|$ for all $x\in\R^n$, $n\ge 2$.

We start by constructing a microstructure in $\R^n$ (Lemma \ref{lemmalamination}), and then truncating to a ball (Lemma \ref{lemmaseqball}).

\begin{lemma}\label{lemmalamination}
Let $\eps_A, \eps_B, \sigma_A, \sigma_B\in\R^{n\times n}_\sym$, $\nu\in S^{n-1}$, $c\in\R^n$ be such that
\begin{equation}
    \eps_A-\eps_B=c\odot\nu \hskip5mm\text{ and }\hskip5mm (\sigma_A-\sigma_B)\nu=0.
\end{equation}
Let $\lambda\in (0,1)$. Then for any $h\in\N$ there are functions $u^h\in W^{1,\infty}_\loc(\R^n;\R^n)$ and $\apot^h\in W^{2,\infty}_\loc(\R^n;\Rvierst)$ such that
\begin{equation}
 e(u^h)=\eps_A \text{ and }\Div\Div \apot^h=\sigma_A \text{ \ae\ in the set } \{x: h x\cdot \nu \in \Z+(0,\lambda)\}
\end{equation}
and
\begin{equation}
 e(u^h)=\eps_B \text{ and }\Div\Div \apot^h=\sigma_B \text{ \ae\ in the set } \{x:  h x\cdot \nu \in \Z+(\lambda,1)\},
\end{equation}
with
$\|Du^h\|_{L^\infty}\le |\eps_A|+|\eps_B|$, $\|D^2\apot^h\|_{L^\infty}\le 2( |\sigma_A|+|\sigma_B|)$,
$\|u^h-\bar u\|_{L^\infty}\le |\eps_A-\eps_B|/h$ and $\|\apot^h-\bar \apot\|_{L^\infty}\le |\sigma_A-\sigma_B|/h^2$,
where
\begin{equation}
\bar u(x)=(\lambda \eps_A+(1-\lambda)\eps_B)x, \hskip5mm\text{ and }\hskip5mm
\bar \apot=\apot^{\lambda \sigma_A+(1-\lambda)\sigma_B},
\end{equation}
as in (\ref{eqdefaM}). The constant $C$ depends only on $\eps_A, \eps_B, \sigma_A, \sigma_B$, $\nu$.
\end{lemma}

\begin{proof}
We define $\bar\eps=\lambda \eps_A+(1-\lambda)\eps_B$ and $\bar\sigma=\lambda \sigma_A+(1-\lambda)\sigma_B$.

The construction of $u^h$ is standard, see for example \cite{Dacorogna1989,MuellerLectureNotes}. Indeed, it suffices to let $f:\R\to\R$ be the Lipschitz, 1-periodic function such that $f(0)=0$, and $f'(x)=1-\lambda$ on $(0,\lambda)$, $f'(x)=-\lambda$ on $(\lambda,1)$, and set
\begin{equation}
u^h(x) = \bar\eps x + c \frac1h f(hx\cdot \nu).
\end{equation}
Then $\|f\|_{L^\infty}\le1$, $\|f'\|_{L^\infty}\le 1$ and $e(u^h)(x)=\bar\eps + (\eps_A-\eps_B) f'(hx\cdot\nu)$ give the result.

In order to construct $\apot^h$, we start by showing that there exists a matrix $\xi\in\Rvierst$, such that
\begin{equation}\label{eqbsigma}
\xi_{ijhk}\nu_h\nu_k = (\sigma_A-\sigma_B)_{ij}.
\end{equation}
or, equivalently, $\xi(\nu\otimes \nu)=\sigma_A-\sigma_B$. Indeed, it suffices to take
\begin{equation}
 \xi_{ijhk}=\hat\sigma_{ij}\nu_h\nu_k + \hat\sigma_{hk}\nu_i\nu_j - \hat\sigma_{ih}\nu_j\nu_k -\hat\sigma_{jk}\nu_i\nu_h ,
\end{equation}
where $\hat\sigma=\sigma_A-\sigma_B$. From $\hat\sigma^T=\hat\sigma$ (\ref{eqdefRvierst}) follows, from $\hat\sigma\nu=0$ (\ref{eqbsigma}) follows. It is also clear that $|\xi|\le 4|\sigma_A-\sigma_B|$.

We set
\begin{equation}
  \apot^h_{ijhk}(x) =
  \bar \apot_{ijhk}(x)
  + \xi_{ijhk} \frac1{h^2} F(hx\cdot \nu) ,
\end{equation}
where $F$ is a primitive of $f$ with average 0 over one period, and compute $\|D^2\apot^h\|_{L^\infty} \le |\bar\sigma|+|\xi|$, $\|\apot^h-\bar \apot\|_{L^\infty} \le |\xi|/{h^2}$,
\begin{equation}
\Div\Div \apot^h = \bar\sigma + (\sigma_A-\sigma_B) f'(hx\cdot \nu).
\end{equation}
Inserting the definition of $\bar\sigma$ and the two values of $f'$ gives the result.
\end{proof}

Given a bounded Lipschitz set $\omega\subset\R^n$, we write
\begin{equation}
\begin{split}
    \setE_0(\omega)
    = &
    \{ (\eps,\sigma)\in L^2(\omega;\R^{n\times n}_\sym\times \R^{n\times n}_\sym):
    \\ &
    \Div\sigma=0 \text{ and } \exists u \in H^1(\omega;\R^n): \eps=e(u)\}.
\end{split}
\end{equation}

\begin{lemma}\label{lemmaseqball}
Let $r>0$, $(\bar\eps,\bar\sigma)\in \overline \setD_\loc$. Then, there are sequences $(\alpha_h,\beta_h)\in L^\infty(B_r;\setD_\loc)$ and $(\eps_h,\sigma_h)\in \setE_0(B_r)$ such that
\begin{equation}\label{eqlemmaseqballstr}
\eps_h-\alpha_h \to 0 \hskip5mm\text{and}\hskip5mm \sigma_h-\beta_h\to0 \hskip5mm \text{ strongly in } L^2(B_r),
\end{equation}
\begin{equation}
 \eps_h\weakto\bar\eps  \hskip5mm\text{and}\hskip5mm \sigma_h\weakto\bar\sigma \hskip5mm \text{ weakly in } L^2(B_r),
\end{equation}
with the bounds $\|\eps_h\|_{L^\infty}\le C (|\bar\eps|+1)$, $\|\sigma_h\|_{L^\infty}\le C (|\bar\sigma|+1)$. The constant depends only on $\C$, $a$, $b$ (\ie, on the set $\setD_\loc$). Furthermore, there is  a sequence $\eta_h\to0$, $\eta_h>0$, such that
\begin{equation}\label{eqbwlemmasqball}
\sigma_h=\bar\sigma \text{ and } \eps_h=\bar\eps \text{ on } B_r\setminus B_{(1-\eta_h) r}.
\end{equation}
\end{lemma}
\begin{proof}
If $(\bar\eps,\bar\sigma)\in \setD_\loc$ then the constant sequences $\alpha_h = \eps_h = \bar\eps$, $\beta_h=\sigma_h=\bar\sigma$ suffices.

Let $z_- = (\eps_A,\sigma_A), z_+= (\eps_B,\sigma_B)\in \setD_\loc$,  $\lambda\in(0,1)$, $c\in\R^n$, $\nu\in S^{n-1}$ be as given by Lemma  \ref{le:rank_one_connection}.
Let $u_h$, $\apot_h$ be the corresponding sequences from Lemma \ref{lemmalamination}.
For any $\eta\in (0,1/2)$, we choose $\theta_\eta\in C^1_c(B_{(1-\eta)r};[0,1])$ such that $\theta_\eta=1$ on $B_{(1-2\eta)r}$ and $\|D\theta_\eta\|_{L^\infty}\le 3/(\eta r)$. We define
\begin{equation}
v_h = u_h\theta_{1/h} + \bar u (1-\theta_{1/h})
\end{equation}
and
\begin{equation}
  \bpot_h = \apot_h\theta_{1/h} + \bar \apot (1-\theta_{1/h}),
\end{equation}
where $\bar u$, $\bar \apot$ are the affine potentials defined in Lemma \ref{lemmalamination}. We then set
\begin{equation}
 \eps_h= e(v_h), \hskip5mm \sigma_h=\Div\Div \bpot_h
\end{equation}
and
\begin{equation}
 \alpha_h= e(u_h), \hskip5mm \beta_h=\Div\Div \apot_h.
\end{equation}
It is easy to see that $(\eps_h,\sigma_h)\in \setE_0(B_r)$ and that they obey (\ref{eqbwlemmasqball}) with $\eta_h=1/h$. By Lemma \ref{lemmalamination}, we also have $(\alpha_h,\beta_h)\in \{(\eps_A, \sigma_A),(\eps_B,\sigma_B)\}\subset \setD_\loc$ almost everywhere. The sequences $(u_h, \apot_h)$ converge uniformly to $(\bar u,\bar \apot)$.

Furthermore,
\begin{equation}
\|\alpha_h\|_{L^\infty}\le \|\eps_h\|_{L^\infty} \le |\eps_A|+|\eps_B| + \|u_h-\bar u\|_{L^\infty} \|D\theta_{1/h}\|_{L^\infty}
\le c (|\eps_A|+|\eps_B|)
\end{equation}
implies that $\eps_h$ and $\alpha_h$ have a weak-* limit in $L^\infty$; since $u_h\to\bar u$ we obtain $\alpha_h\weakstarto \bar \eps$ in $L^\infty$. At the same time, since $\eps_h=\alpha_h$ on $B_{(1-2\eta_h)r}$ we obtain
\begin{equation}
\begin{split}
 \|\eps_h-\alpha_h\|_{L^2(B_r)}^2 &\le |B_r\setminus B_{(1-2\eta_h)r}| (\|\eps_h\|_{L^\infty}+\|\alpha_h\|_{L^\infty} )^2
 \\
 &\le C(|\eps_A|+|\eps_B|)^2 \eta_h |B_r| \le C (|\bar \eps|^2+1) |B_r|/h
 \end{split}
\end{equation}
where in the last step we used \eqref{eq:laminate_bound}. Therefore, $\eps_h-\alpha_h\to0$ strongly in $L^2$ and $\eps_h\weakto \bar\eps$ in $L^2$.

Similarly, from
\begin{equation}
\begin{split}
\|\beta_h\|_{L^\infty}\le \|\sigma_h\|_{L^\infty} &\le |\sigma_A|+|\sigma_B| + \|D\apot_h-D\bar \apot\|_{L^\infty} \|D\theta_{1/h}\|_{L^\infty} \\
&+ \|\apot_h-\bar \apot\|_{L^\infty} \|D^2\theta_{1/h}\|_{L^\infty}
\le c (|\sigma_A|+|\sigma_B|)
\end{split}
\end{equation}
we obtain that $\beta_h$ and $\sigma_h$ have a weak-* limit in $L^\infty$, with $\beta_h \weakstarto \bar\sigma$. Furthermore,
\begin{equation}
\begin{split}
 \|\sigma_h-\beta_h\|_{L^2(B_r)}^2 &\le |B_r\setminus B_{(1-2\eta_h)r}| (\|\sigma_h\|_{L^\infty}+\|\beta_h\|_{L^\infty} )^2\\
 &\le C(|\sigma_A|+|\sigma_B|)^2 \eta_h |B_r|
 \le C (|\bar \sigma|^2+1) |B_r|/h.
\end{split}
\end{equation}
Therefore, $\sigma_h-\beta_h\to0$ strongly in $L^2$.
\end{proof}

The proof of the upper bound proceeds by local modification around Lebesgue points of the fields $\bar\eps$, $\bar\sigma$, as was done in \cite{ContiDolzmann2015}. This sidesteps the need to go through a density argument. We present in Lemma \ref{lemmaseqball2} the construction in a ball around a Lebesgue point, and then in the proof of Theorem \ref{theoub} the covering argument.

We define, given a bounded Lipschitz set $\Omega$ and  $f\in L^{2}(\Omega;\R^n)$,
\begin{equation}
 \begin{split}
  \setE_f=\{& (\eps,\sigma)\in L^2(\Omega;\R^{n\times n}_\sym\times \R^{n\times n}_\sym): \Div\sigma=f \\
  &\text{ and } \exists u \in H^1(\Omega;\R^n): \eps=e(u)\}.
 \end{split}
\end{equation}

\begin{lemma}\label{lemmaseqball2}
Let $(\bar\eps,\bar\sigma)\in L^2(\Omega;\overline {\setD}_\loc)\cap \setE_f$, $B = B_r(x_*)\subset\Omega$ be a ball such that
\begin{equation}
  (\bar\eps(x_*),\bar\sigma(x_*))\in \overline \setD_\loc
\end{equation}
and, for some $\delta\in(0,1)$,
\begin{equation}\label{eqlemmaseqball2deltaass}
   \frac{1}{|B|} \int_{B} |\bar\eps-\bar\eps(x_*)|^2+|\bar\sigma-\bar\sigma(x_*)|^2dx<\delta.
\end{equation}
Then, there are sequences $(\alpha_h,\beta_h)\in L^\infty(B_r;\setD_\loc)$ and $(\eps_h^*,\sigma_h^*)\in \setE_f$ such that $\{\eps_h^*\ne\bar\eps\}\cup \{\sigma_h^*\ne\bar\sigma\}\subset\subset B_r$,
\begin{equation}\label{eqlemmaseglimsup}
\limsup_{h\to\infty}\left(\|\eps_h^*-\alpha_h\|_{L^2(B_r)}^2+\|\sigma_h^*-\beta_h\|_{L^2(B_r)}^2\right)\le \delta |B|,
\end{equation}
\begin{equation}\label{lemmaseqball2weakconv}
 \eps_h^*\weakto\bar\eps  \hskip5mm\text{and}\hskip5mm \sigma_h^*\weakto\bar\sigma \hskip5mm \text{ weakly in } L^2(B_r),
\end{equation}
with the bound
\begin{equation}\label{eqlemmaunifbd}
\limsup_{h\to\infty}\left(\|\eps_h^*\|_{L^2(B_r)}^2+\|\sigma_h^*\|_{L^2(B_r)}^2\right)\le  C (\|\bar\eps\|_{L^2(B_r)}^2+\|\bar\sigma\|_{L^2(B_r)}^2+|B_r|).
\end{equation}
The constant $C$ depends only on the set $\setD_\loc$.
\end{lemma}
\begin{proof}
Let $(\eps_h,\sigma_h,\alpha_h,\beta_h)$ be the sequences from Lemma \ref{lemmaseqball}, applied on the ball $B_r$ with the matrices $\bar\eps(x_*)$, $\bar\sigma(x_*)$. We define
\begin{equation}
  \eps_h^*=\bar\eps + \eps_h-\bar\eps(x_*) \hskip5mm\text{ and }\hskip5mm
  \sigma_h^*=\bar\sigma + \sigma_h-\bar\sigma(x_*) \hskip5mm \text{ in $B_r$,}
\end{equation}
and $\eps_h^*=\bar\eps$, $\sigma_h^*=\bar\sigma$ in $\Omega\setminus B_r$. We claim that the sequences $(\eps_h^*,\sigma_h^*,\alpha_h,\beta_h)$ have the required properties. Indeed, $\Div \sigma_h^*=\Div\bar\sigma$. Let $\bar u$ be as in the definition of $\setE_f$, so that in particular $\bar\eps=e(\bar u)$. From $(\eps_h,\sigma_h)\in \setE_0(B_r)$ and (\ref{eqbwlemmasqball}) we deduce that  there is $v_h\in H^1(B_r)$ such that $\eps_h=e(v_h)$, and $v_h(x)=\bar\eps(x_*) x$ around $\partial B$. We set $\bar u_h^*(x)=\bar u(x) + v_h(x)-\bar\eps(x_*) x$, so that $\{\bar u_h^*\ne\bar u\}\subset\subset B_r$ and $e(\bar u_h^*)=\eps_h^*$. Therefore, $(\eps_h^*, \sigma_h^*)\in \setE_f$.

From $\eps_h\weakto\bar\eps(x_*)$, we deduce $\eps_h^*\weakto\bar\eps$, and correspondingly for $\sigma_h^*$. Condition (\ref{eqlemmaseglimsup}) follows from (\ref{eqlemmaseqballstr}) and (\ref{eqlemmaseqball2deltaass}). The condition (\ref{eqlemmaunifbd}) follows from the corresponding condition in Lemma \ref{lemmaseqball}.
\end{proof}

\begin{theorem}     \label{theoub}
Let $(\bar\eps,\bar\sigma)\in L^2(\Omega;\overline \setD_\loc)\cap \setE_f$. Then there are sequences $(\eps_h,\sigma_h)\in \setE_f$ and $(\alpha_h,\beta_h)\in L^2(\Omega;\setD_\loc)$ such that
\begin{equation}
 \eps_h-\alpha_h \to 0 \hskip5mm\text{and}\hskip5mm \sigma_h-\beta_h\to0 \hskip5mm \text{ strongly in } L^2(\Omega),
\end{equation}
\begin{equation}
 \eps_h\weakto\bar\eps  \hskip5mm\text{and}\hskip5mm \sigma_h\weakto\bar\sigma \hskip5mm \text{ weakly in } L^2(\Omega) ,
\end{equation}
with $\{\eps_h\ne \bar\eps\}\cup\{\sigma_h\ne\bar\sigma\}\subset\subset\Omega$ for any $h$.
\end{theorem}

\begin{proof}
We define
\begin{align}
 \omega=\big\{x\in\Omega:&
 (\bar\eps(x),\bar\sigma(x))\in \overline \setD_\loc \text{ and }\\
& \liminf_{r\to0} \frac{1}{|B_r|} \int_{B_r(x)} |\bar\eps-\bar\eps(x)|^2+|\bar\sigma-\bar\sigma(x)|^2dx=0\big\}.
\end{align}
By the Lebesgue point theorem, $\calL^n(\Omega\setminus\omega)=0$. We further choose (arbitrarily) a pair
$(\eps_0,\sigma_0)\in \setD_\loc$ and define
\begin{equation}\label{eqdefetadelta}
 \eta(\delta)=\sup \left\{ \int_F |\bar\eps-\eps_0|^2+|\bar\sigma-\sigma_0|^2 dx: F\subset \Omega, |F|\le \delta\right\};
\end{equation}
by the continuity of the integral we have $\eta(\delta)\to0$ as $\delta\to0$.

Fix $\delta>0$. By Vitali's covering theorem, we can cover $\calL^n$-almost all of $\omega$ by countably many pairwise disjoint balls $B_i=B(x_i, r_i)\subset\subset\Omega$ with the property
\begin{equation}\label{eqgoodball}
   (\bar\eps(x_i),\bar\sigma(x_i))\in \overline \setD_\loc \hskip3mm \text{ and } \hskip3mm\frac{1}{|B_i|} \int_{B_i} |\bar\eps-\bar\eps(x_i)|^2+|\bar\sigma-\bar\sigma(x_i)|^2dx<\delta.
\end{equation}
and we can have a finite set $B_1,\dots, B_M$ of balls that obey (\ref{eqgoodball}) and $\calL^n(\Omega\setminus \cup_{i=1}^M B_i)<\delta$.

We define $(\eps_h^\delta, \sigma_h^\delta, \alpha_h^\delta, \beta_h^\delta)$, as in the assertion of Lemma \ref{lemmaseqball2}, in each of the balls $B_1,\dots, B_M$, and as $(\bar\eps,\bar\sigma,\eps_0,\sigma_0)$ in $\Omega\setminus \cup_{i=1}^M B_i$; the pair $(\eps_0,\sigma_0)\in \setD_\loc$ was chosen before (\ref{eqdefetadelta}).

Then, it is easily verified that $(\eps_h^\delta,\sigma_h^\delta)\in \setE_f$, $(\alpha_h^\delta, \beta_h^\delta)\in L^2(\Omega;\setD_\loc)$, and
\begin{equation}
\begin{split}
    &
 \limsup_{h\to0} \left(\|\eps_h^\delta-\alpha_h^\delta\|_{L^2(\Omega)}+\|\sigma_h^\delta-\beta_h^\delta\|_{L^2(\Omega)}\right)
 \le \\ & \sum_{i=1}^M |B_i| \delta + C \eta(\delta) \le \delta |\Omega|+\eta(\delta).
\end{split}
\end{equation}
At the same time, from (\ref{lemmaseqball2weakconv}) we obtain that $\eps_h\weakto\bar\eps$ and $\sigma_h\weakto\bar\sigma$. From (\ref{eqlemmaunifbd}), we easily see that the sequences are bounded in $L^2$ uniformly in $\delta$. Therefore, we can take a diagonal subsequence and conclude the proof.
\end{proof}

\bigskip

We now show that the set $\overline \setD \times \setE$ is closed under Data convergence.

\begin{theorem}    \label{theotwowellllwerb}
Let  $(\eps_h,\sigma_h)\in  \setE_f$ and $(\alpha_h,\beta_h)\in L^2(\Omega;\overline \setD_\loc)$ be such that
\begin{equation}
 \eps_h-\alpha_h \to 0 \hskip5mm\text{and}\hskip5mm \sigma_h-\beta_h\to0 \hskip5mm \text{ strongly in } L^2(\Omega).
\end{equation}
Assume that
\begin{equation}
 \eps_h\weakto\bar\eps  \hskip5mm\text{and}\hskip5mm \sigma_h\weakto\bar\sigma \hskip5mm \text{ weakly in } L^2(\Omega)
\end{equation}
for some $\bar\eps,\bar\sigma\in L^2(\Omega;\R^{n\times n}_\sym\times\R^{n\times n}_\sym)$. Then,
\begin{equation}
(\bar\eps,\bar\sigma)\in L^2(\Omega; \overline \setD_\loc)\cap \setE_f.
\end{equation}
\end{theorem}

We remark that this result has the typical form of a lower semicontinuity statement. It is equivalent to the assertion that if $z_h=(\eps_h,\sigma_h)\in \setE_f$ converge weakly to $z=(\eps,\sigma)$ and
\begin{equation}
    \lim_{h\to\infty} \int_\Omega d^2(z_h, \bar\setD_\loc) dx=0,
\end{equation}
then $z_h\in \bar\setD_\loc$ almost everywhere. This assertion is in turn equivalent to the requirement that the set $\bar\setD_\loc$ be $\calA$-quasiconvex, where $\calA$ is the differential operator corresponding to the condition $z_h\in \setE_f$. This property of data sets can therefore be studied within the framework of $\calA$-quasiconvexity developed in \cite{FonsecaMueller1999}, though care needs to be exercised due to the fact that in the present case the differential operator has mixed order, see discussion in \cite[Example 3.10(b) and (e)]{FonsecaMueller1999}. By virtue of this connection, the general tools from \cite{FonsecaMueller1999} may be used instead of the explicit constructions in Lemma \ref{lemmaseqball} and Lemma \ref{lemmaseqball2}. In particular, the Data-hull $\bar\setD_\loc$ can be abstractly characterized
%% SM Aug 8 new
by infimizing over periodic pairs $(\eps, \sigma)$ with average zero and the constraints $\eps = e(u)$ and $\Div \sigma=0$.
% as the set of $z\in Z$ such that
%\begin{equation}
%\inf_{u\in H^1(\T^n; \R^n), \sigma\in L^2_\Div(\T^n;\R^{n\times n}_\sym)}
%\int_{\T^n} d^2(z+(e(u),\sigma), \setD_\loc) dx =0
%\end{equation}
%where $\T^n=(0,1)^n$ is the unit torus and periodic boundary conditions are
%implicitly assumed.
However, a detailed elucidation of $\calA$-quasiconvexity
as it applies to Data-Driven problems is beyond the scope of the this paper
%%SM Aug 8,  reinserted the text below
and we will pursued in future work.
%\blue{We will pursue this in future work.}

\bigskip

To prove Theorem \ref{theotwowellllwerb}, we first identify suitable quadratic functions which are lower semicontinuous under the convergence assumptions in the theorem. Then we use these quantities to show that points outside the set $\overline \setD_\loc$, given by \eqref{eqdefDlocbar3_bis}, cannot arise as weak limits.

By the theory of compensated compactness the constraints $\eps = e(u)$ and $\Div \sigma =f$ guarantee (in combination with Korn's inequaltiy) that quadratic functionals which are nonnegative on the cone $\Lambda$ defined in  \eqref{eq:lambda_cone} are weakly sequentially lower semicontinuous. We know in addition that the sequence $(\eps_h, \sigma_h)$ approaches the set $\overline \setD_\loc$ and hence the space $L$ strongly. This allows us to use  quadratic quantities that are nonnegative only on $\Lambda \cap L$. This reduction is standard (and in fact essentially contained in Tartar's proof of lower semicontinuity by Fourier transform, see \cite{Tartar1979}, Theorem 11) but we recall the statement and the proof for the convenience of the reader.

\begin{lemma} \label{le:separate_on_L_cap_Lambda}
Recall from \eqref{eq:lambda_cone} that 
\begin{equation*} 
\Lambda = \{ (\eps, \sigma) \in \R^{n \times n}_\sym \times \R^{n \times n}_\sym : \eps = c \odot \nu, \, \sigma \nu = 0, \, c \in \R^n, \, \nu \in S^{n-1} \}.
\end{equation*}
Let $L \subset \R^{n \times n}_\sym \times \R^{n \times n}_\sym$ be a linear subspace. Assume that $Q:  \R^{n \times n}_\sym \times \R^{n \times n}_\sym \to \R$ is a quadratic form and
$$
    Q_{|\Lambda \cap L} \ge 0.
$$
Let $\Omega \subset \R^n$ be open. Assume that $f \in L^2(\Omega;\R^n)$
$$
    \eps_h = e(u_h), \quad \Div \sigma_h = f,
$$
$$
    (\eps_h, \sigma_h) \weakto (\bar \eps, \bar \sigma)
    \quad
    \hbox{in  $L^2(\Omega; \R^{n \times n}_\sym \times \R^{n \times n}_\sym)$}
$$
and
$$
    d((\eps_h, \sigma_h), L) \to 0 \quad \hbox{in $L^2(\Omega)$}.
$$
Then,
\begin{equation}  \label{eq:lsc_on_Lambda_cap_L}
 \int_\Omega Q (\bar \eps, \bar \sigma)  \, \varphi \, dx  \le \liminf_{h \to \infty} \int_\Omega Q (\eps_h, \sigma_h) \, \varphi \, dx \quad \forall \varphi \in C_c(\Omega), \, \, \varphi \ge 0.
\end{equation}
\end{lemma}

\begin{proof}
We begin with some general comments. It suffices to prove the statement for  $\Omega$ being a ball. The general case then follows by a partition of unity. Since linear bounded functionals are weakly continuous, it also suffices to prove the result for the special case
\begin{equation} \label{eq:cc_zero_limit}
(\bar \eps, \bar \sigma) = 0 \quad \hbox{and} \quad f=0.
\end{equation}
To get the full result one can apply the special case to the sequence $(\eps_h - \bar \eps, \sigma_h - \bar \sigma)$.

To reduce the assertion to the standard result in compensated compactness, we first show that (up to an arbitrarily small error) $Q$ can  be extended to a quadratic form on $\R^{n \times n}_\sym \times \R^{n \times n}_\sym$ which is nonnegative on $\Lambda$. More precisely, we write $z = (\eps, \sigma)$ and   define
$$
    Q_{\delta,M}(z) := Q(z) + M |P_{L^\perp} z|^2 + \delta |z|^2 ,
$$
where $P_L$ denotes the orthogonal projection onto $L$. We claim that
\begin{equation} \label{eq:extension_Q}
\forall \delta > 0 \, \, \exists M > 0 \, \, \forall z \in \Lambda  \quad Q_{\delta,M}(z) \ge 0
\end{equation}
Assume that this was false. Then, there exists a $\delta_0 > 0$ and sequence $M_k \to \infty$ and $z_k \in \Lambda$ such that
\begin{equation}  \label{eq:extension_contradiction}
 Q(z_k) + M_k |P_{L^\perp} z_k|^2 + \delta_0 |z_k|^2 < 0.
\end{equation}
Since the left hand side is homogeneous of degree $2$, we may assume that $|z_k| = 1$. Passing to a subsequence (not renamed) we may assume that $z_k \to \bar z$ in $\subset \R^{n \times n}_\sym \times \R^{n \times n}_\sym$. Then $|\bar z| =1$ and $\bar z \in \Lambda$ since $\Lambda$ is closed. Dividing \eqref{eq:extension_contradiction} by $M_k$ and passing to the limit, we see that $P_{L^\perp} \bar z = 0$. Thus, $\bar z \in L$ and  \eqref{eq:extension_contradiction} implies that $Q(\bar z) + \delta_0 = Q(\bar z) + \delta_0 |\bar z|^2  \le 0$. Hence, $Q(\bar z) < 0$ and $\bar z \in \Lambda \cap L$. This contradiction finishes the proof of  \eqref{eq:extension_Q}.

Now, we prove \eqref{eq:lsc_on_Lambda_cap_L} under the assumption \eqref{eq:cc_zero_limit} for the case that  $\Omega$ is a ball. Fix $\varphi \in C_c(\Omega)$ with $\varphi \ge 0$.
Let
$$
    R = \limsup_{h \to \infty} \| (\eps_h, \sigma_h)\|_{L^2}^2.
$$
Since weakly convergent sequences are bounded, we have $R < \infty$. The constraint $\eps_h = e(u_h)$ is (locally) equivalent to the second order constraint $\curl^T \curl \eps_h = 0$. Since the classical results for compensated compactness are formulated for first order conditions, we work with $\nabla u_h$ rather than $\eps_h$.

By Korn's inequality there exist $u_h \in H^1(\Omega;\R^n)$ such that
$\eps_h = e(u_h)$ and
$$
    (\nabla u_h, \sigma_h) \weakto 0  \quad \hbox{in $L^2(\Omega; \R^{n \times n} \times \R^{n \times n}_\sym)$.}
$$
Let $\tilde \Lambda$ be the cone corresponding to the constraints
$\curl F_h = 0$  and $\Div \sigma_h = 0$, \ie,
$$
    \tilde \Lambda := \{ (F, \sigma) \in \R^{n \times n} \times \R^{n \times n}_\sym :
    F = c \otimes \nu, \, \sigma \nu = 0, \, c\in \R^n, \, \nu \in S^{n-1} \}.
$$
Let $\delta >0$ and let $M$ be such that $Q_{\delta,M}$ is nonnegative on $\Lambda$. Extend $Q_{\delta,M}$ trivially to  $\R^{n \times n} \times \R^{n \times n}_\sym$ by
$$
    \tilde Q_{\delta,M}(F, \sigma)  =Q_{\delta, M}( \mathop{\sym} F, \sigma).
$$
Then, $\tilde Q_{\delta, M} \ge 0$ on $\tilde \Lambda$ and the theory of compensated compactness gives (see \cite{Tartar1979}, Theorem 11)
\begin{equation}
0  \le  \liminf_{h \to \infty} \int_\Omega
\tilde Q_{\delta,M} (\nabla u_h, \sigma_h) \, \varphi \, dx
=  \liminf_{h \to \infty} \int_\Omega
Q_{\delta,M}(\eps_h, \sigma_h) \, \varphi \, dx
\end{equation}
Since $d((\eps_h, \sigma_h), L) \to 0$ strongly in $L^2$, we see that
$$
    M  |P_{L^\perp} (\eps_h, \sigma_h)|^2 \to 0 \quad \hbox{in $L^1(\Omega)$}.
$$
Thus
\begin{equation}
 0  \le \liminf_{h \to \infty} \int_\Omega
Q (\eps_h, \sigma_h) \, \varphi \, dx + \delta R \sup \varphi.
\end{equation}
Since $\delta > 0$ was arbitrary, the desired assertion \eqref{eq:lsc_on_Lambda_cap_L}
follows (recall that $(\bar \eps, \bar \sigma) = 0$).
\end{proof}

\begin{figure}[ht]  
\begin{tikzpicture}
%\path [use as bounding box] (-7, -2)  rectangle (12,10);
\draw[->][style=thin] (0,-2) -- (0,3) node[anchor=west]    {$\mu$};
\draw[->][style=thin] (-6,0) -- (6,0) node[anchor=north]    {$\sigma \cdot b$};
\draw[style=thick]  (-6,-1) -- (2,-1);
\draw[style=thick]  (-2,1) -- (6,1);
\draw[style=dashed]  (2,-1) -- (1,1);
\draw[style=dashed] (-1,-1) -- (-2,1);
\draw[color=blue] (2,2) -- (3.75,-1.5); % through (2.5,1)
\draw[color=red] (1.5,2)--(5,-1.5);
\fill (2.5,1) circle (0.07)  node[anchor=south] {$P_*$};
\fill (2.75, .5) circle(0.07) node[anchor=east] {$P_0$};
\end{tikzpicture}
\caption{Illustration of the set $\overline \setD_\loc$ and the separation construction.
The solid black half-lines represent $\setD_\loc$. The set $\overline \setD_\loc$ is obtained by adding the parallelogram bounded  by the solid and dashed segments. The dashed segments lie on lines on which $\sigma \cdot b + \alpha_- \mu$ is constant. The separation of a point $P_0 \notin \overline \setD_\loc$ is 
indicated by the blue and red lines. 
The blue line is the  line through $P_0$ on which $\sigma \cdot b + \alpha_- \mu$ is constant. 
The intersection point with the line $\mu =1$ is $P_* = (\sigma_* \cdot b, 1)$. 
The red line is the line through $P_*$ on which $\sigma \cdot b + \alpha_+ \mu$ is constant. 
The function $Q(\eps- \eps_*, \sigma- \sigma_*)$ is positive on the narrow  
region between the red and blue lines and negative 
on the wide region between the red and blue line. It vanishes on the blue and red line. 
In particular, this function is strictly negative in the parallelogram bounded by the solid and shade lines and
on the solid half-line with $\mu = -1$. 
By adding a small multiple of $1 - \mu$, we obtain a  function that is strictly positive at $P_0$ and nonpositive on $\overline \setD_\loc$. } 
\label{fi:separation}
\end{figure}
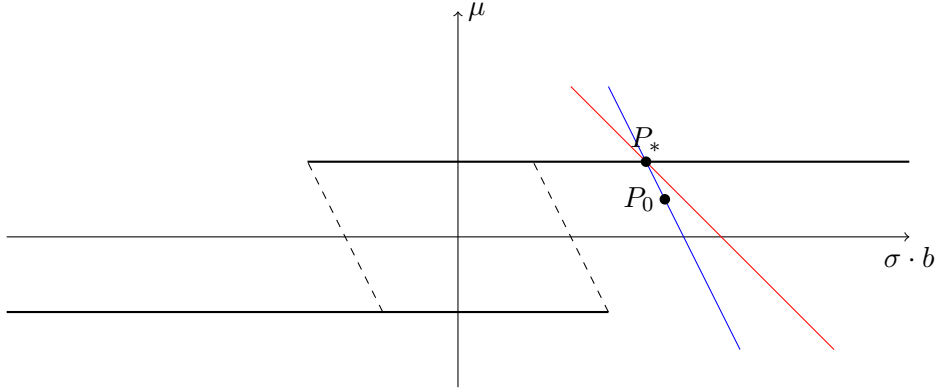

We now consider on $L$ the quadratic form
\begin{equation}  \label{eq:quadratic_separation_1}
 \tilde Q( \C^{-1} \sigma + \mu b, \sigma) = - (\sigma \cdot b + \alpha_-  \mu) (\sigma \cdot b + \alpha_+ \mu)
 \end{equation}
where $\alpha_\pm$ are defined in Lemma \ref{le:Lambda_cap_L}. We extend $Q$ to a quadratic form on $\R^{n \times n}_\sym \times \R^{n \times n}_\sym$ by  setting
\begin{equation} \label{eq:quadratic_separation_2}
Q(z) = \tilde Q(P_L z)
\end{equation}
Then, \eqref{eq:separate_Lambda_1} implies that
\begin{equation}  \label{eq:sep_quadratic_positive}
 Q \ge 0  \quad \hbox{on $\Lambda \cap L$.}
\end{equation}

\begin{lemma}[Separating quadratic functions, see Fig.~\ref{fi:separation}]  \label{le:separate_quadratic}
Let
$$
    U := \left\{ (\C^{-1} \sigma_0 + \mu_0 b, \, \sigma_0) :  \mu_0 \in [-1, 1), \,  \sigma_0 \cdot b + \alpha_- \mu_0 > \C b \cdot b - \alpha_-   \right\} \subset L.
$$
For $(\eps_0, \sigma_0) \in U$, define
$$
    ( \eps_*, \sigma_*) = (\eps_0, \sigma_0) + (1 - \mu_0) \hat z ,
$$
where $\hat z$ is as   \eqref{eq:optimal_Lambda_conn}. Then, $(\eps_*, \sigma_*) \in \setD^+_\loc$ and
\begin{equation}  \label{eq:sigm_*_on_line}
\sigma_* \cdot b + \alpha_- = \sigma_0 \cdot b + \alpha_- \mu_0
\end{equation}
Define
$$
    f_{\eps_0, \sigma_0}(\eps, \sigma) = Q (\eps - \eps_*, \sigma - \sigma_*) + \delta_{\eps_0, \sigma_0} \big[1 - (\eps- \C^{-1} \sigma) \cdot |b|^{-2} b\big] ,
$$
where $Q$ is given by \eqref{eq:quadratic_separation_1} and \eqref{eq:quadratic_separation_2} and
$$
    \delta_{\eps_0, \sigma_0}  = \frac12
    \big[ \sigma_0 \cdot b + \alpha_- \mu_0  - ( \C b \cdot b - \alpha_-)  \big]^2
    > 0.
$$
Then,
\begin{equation}  \label{eq:f_on_barD}
    f_{\eps_0, \sigma_0} \le 0   \quad \hbox{on $\overline \setD_\loc$}
\end{equation}
and
\begin{equation}  \label{eq:f_at_point}
f_{\eps_0, \sigma_0}(\eps_0, \sigma_0) > 0.
\end{equation}
\end{lemma}

\begin{proof}  To prove  \eqref{eq:sigm_*_on_line} note that by  \eqref{eq:optimal_Lambda_conn}
$$ \sigma_* \cdot b + \alpha_- = \sigma_0 \cdot b -(1-\mu_0) \alpha_- + \alpha_- = \sigma_0 \cdot b + \alpha_- \mu_0.$$
For $(\eps, \sigma) \in L$, we have $\eps = \C^{-1} \sigma + \mu b$ and, thus,
\begin{align*} Q(\eps-\eps_*, \sigma- \sigma_*) =&\,  - \big [\sigma \cdot b + \alpha_- \, \mu - (\sigma_* \cdot  b+ \alpha_-)\big]
\big[(\sigma \cdot b + \alpha_+ \, \mu - (\sigma_* \cdot b+ \alpha_+)\big], \\
1 - (\eps- \C^{-1} \sigma) \cdot |b|^{-2} b =&\,  \,  1 - \mu.
\end{align*}
Using  \eqref{eq:sigm_*_on_line}, we see that $Q(\eps_0- \eps_*, \sigma_0 - \sigma_*) = 0$ and, hence,
$$ f_{\eps_0, \sigma_0}(\eps_0, \sigma_0) = \delta_{\eps_0, \sigma_0} (1- \mu_0) > 0.$$
This proves \eqref{eq:f_at_point}.

To show  \eqref{eq:f_on_barD}, we first note that for $(\eps, \sigma) \in \setD^+_\loc$ we have $\mu=1$ and, thus,
$$
    \forall (\eps, \sigma) \in \setD^+_\loc \quad f_{\eps_0, \sigma_0}(\eps, \sigma) = Q(\eps-\eps_*, \sigma-\sigma_*) = - [\sigma \cdot b - \sigma_* \cdot b]^2 \le 0.
$$
Finally, for $(\eps, \sigma) \in \overline \setD_\loc \setminus \setD^+_\loc$ we have
$\mu \in [-1, 1)$ and
$$
    \sigma \cdot b + \alpha_- \mu \le \C b \cdot b - \alpha_-.
$$
Thus, by \eqref{eq:sigm_*_on_line} and the definition of $U$
$$
    \sigma \cdot b + \alpha_- \, \mu - (\sigma_* \cdot  b+ \alpha_-) \le    \C b \cdot b - \alpha_-       - (\sigma_0 \cdot  b+ \alpha_- \mu_0) < 0
$$
and
$$
    \sigma \cdot b + \alpha_+ \, \mu - (\sigma_* \cdot  b+ \alpha_+) \le  \sigma \cdot b + \alpha_- \, \mu - (\sigma_* \cdot  b+ \alpha_-) < 0 ,
$$
where we used that $\alpha_+ \ge \alpha_-$ and  $1- \mu > 0$. It follows that
$$
    \forall (\eps, \sigma) \in \overline \setD_\loc \setminus \setD^+_\loc \quad
    f_{\eps_0, \sigma_0}(\eps, \sigma) \le - \big[   (\C b \cdot b - \alpha_-)  - (\sigma_0 \cdot  b+ \alpha_- \mu_0)\big]^2 + 2 \delta_{\eps_0, \sigma_0} \le 0
$$
by the definition of $\delta_{\eps_0, \sigma_0}$. This finishes proof of   \eqref{eq:f_on_barD}.
\end{proof}

\begin{proof}[Proof of Theorem  \ref{theotwowellllwerb}]
By convexity, $(\bar \eps, \bar \sigma) \in \setE_f$ and $\bar \eps - \C^{-1} \bar \sigma = \bar \mu b$ \ae \ with $|\bar \mu| \le 1$ \ae \ We will show that there exists a nullset $N$  such that
\begin{equation} \label{eq:separate_thm_1}
\quad \forall\,  x \in \Omega \setminus N
\quad  (\bar \eps(x), \bar \sigma(x)) \in \setD^+_\loc  \quad \hbox{or} \quad  \bar \sigma(x) \cdot b + \alpha_- \, \mu(x) \le \C b \cdot b - \alpha_- \, .
\end{equation}
Applying this result  to the sequences $(-\eps_h, -\sigma_h)$, $(-\alpha_h, -\beta_h)$ and the set $\setE_{-f}$, we get the existence of a null set $N'$ such that
\begin{equation} \label{eq:separate_thm_2}
\quad \forall\,  x \in \Omega \setminus N
\quad  (\bar \eps(x), \bar \sigma(x)) \in \setD^-_\loc  \quad \hbox{or} \quad  \bar \sigma(x) \cdot b + \alpha_- \, \mu(x) \ge -\C b \cdot b + \alpha_-\, .
\end{equation}
The combination of \eqref{eq:separate_thm_1} and \eqref{eq:separate_thm_2} shows that $(\bar \eps, \bar \sigma) \in \overline \setD_\loc$ \ae

To prove \eqref{eq:separate_thm_1} we first fix $(\eps_0, \sigma_0)$ as in Lemma  \ref{le:separate_quadratic} and consider the function $f_{\eps_0, \sigma_0}$. Then, $f_{\eps_0, \sigma_0}(\alpha_h, \beta_h)  \le 0$ by   \eqref{eq:f_on_barD}. Since $f$ consists of linear and quadratic terms the strong convergence of $\eps_h - \alpha_h$ and $\sigma_h - \beta_h$ implies that
$$
    \limsup_{h \to \infty} \int_\Omega f_{\eps_0, \sigma_0}(\eps_h, \sigma_h) \, \varphi \, dx \le 0  \quad \forall \varphi \in C_c(\Omega) \, \, \varphi \ge 0.
$$
Thus Lemma \ref{le:separate_on_L_cap_Lambda} yields
$$
    \int_\Omega f_{\eps_0, \sigma_0}(\bar \eps, \bar \sigma) \, \varphi \, dx \le 0
    \quad \forall \varphi \in C_c(\Omega) \, \, \varphi \ge 0.
$$
Hence, there exists a null set $N_{\eps_0, \sigma_0}$ such that
$$
    f_{\eps_0, \sigma_0}(\bar \eps, \bar \sigma) \le 0
    \quad \hbox{in $\Omega \setminus  N_{\eps_0, \sigma_0}$.}
$$
Considering a countable dense set of points in the set $U$ in Lemma \ref{le:separate_quadratic} and using continuity of $(\eps_0, \sigma_0) \to f_{\eps_0, \sigma_0}(\bar \eps(x), \bar \sigma(x))$, we see that there is a single null set $N$ such that
$$
    \forall (\eps_0, \sigma_0) \in U    \quad \forall x \in \Omega\setminus N \quad  f_{\eps_0, \sigma_0}(\bar \eps(x), \bar \sigma(x)) \le 0.
$$
It now follows from \eqref{eq:f_at_point} that $(\bar \eps(x), \bar \sigma(x)) \notin U$ for all $x \in \Omega \setminus N$. Hence, for all $x \in \Omega\setminus N$
$$
    \bar \sigma(x) \cdot b + \alpha_- \, \bar \mu(x) \le \C b \cdot b - \alpha_- \quad \hbox{or} \quad  \bar \sigma(x) \cdot b + \alpha_- \, \bar \mu(x) >  \C b \cdot b - \alpha_- \hbox{ and  } \bar \mu(x) =1.
$$
In the second case we get $\bar \sigma(x)  \cdot b > \C b \cdot b - 2 \alpha_-$ and $\bar \mu(x) = 1$. Since $\alpha_- \le \C b \cdot b$, it follows that $\bar \sigma \cdot b \ge - \C b \cdot b$ and, thus, $(\bar \eps(x), \bar \sigma(x)) \in \setD^+_\loc$. This finishes the proof of \eqref{eq:separate_thm_1} and, hence, the proof of the theorem.
\end{proof}

\section*{Acknowledgments}

This work was partially supported
by the Deutsche Forschungsgemeinschaft through the Sonderforschungsbereich 1060
{\sl ``The mathematics of emergent effects''}.

\nocite{MuratTartar1997}
\bibliography{biblio}
\bibliographystyle{alpha-noname}

\end{document}